%% file: subvol.tex
\documentclass[12pt]{amsart}

\usepackage{graphicx,xcolor } 
\usepackage{amsmath,amssymb,amsthm,esint}
\usepackage{url}
\usepackage{soul}
\usepackage[shortlabels]{enumitem}

\DeclareMathOperator{\tr}{tr}
\DeclareMathOperator{\Sec}{Sec}

\DeclareMathOperator{\ric}{Ric}
\DeclareMathOperator{\St}{\widetilde{\Sigma}}
\newcommand{\Mt}{\widetilde{M}}

\DeclareMathOperator{\sym}{Sym}

\DeclareMathOperator{\grad}{grad}
\DeclareMathOperator{\diver}{div}
\DeclareMathOperator{\vol}{Vol}
\DeclareMathOperator{\spec}{spec}
\DeclareMathOperator{\id}{id}
\DeclareMathOperator{\riem}{Riem}

\newcommand{\dd}{\partial}
\newcommand{\mf}{\mathfrak}

\newcommand{\bo}{\mathring{b}}
\newcommand{\wtp}{\widetilde{\Psi}}
\newcommand{\gp}{g^+}
\newcommand{\mfr}{\mathfrak{R}}

\newcommand{\mfl}{\mathfrak{L}}
\newcommand{\mfh}{\mathfrak{H}}

\newcommand{\spheres}{\mathbb{S}}

\newcommand\sbt{
    \raisebox{.4ex}{\scalebox{0.4}{$\bullet$}}
}

\newtheorem{proposition}{Proposition}[section]
\newtheorem{lemma}[proposition]{Lemma}
\newtheorem{theorem}[proposition]{Theorem}
\newtheorem{corollary}[proposition]{Corollary}

\title[Volume renormalization in higher codimension]{Volume renormalization of higher-codimension singular Yamabe spaces}
\author{Sri Rama Chandra Kushtagi}
\address{Dept. of Mathematical Sciences, FO 35\\University of Texas at Dallas\\800 W. Campbell Road\\Richardson, TX 75080}
\email{SriRamaChandra.Kushtagi@utdallas.edu}
\author{Stephen E. McKeown}
\address{Dept. of Mathematical Sciences, FO 35\\University of Texas at Dallas\\800 W. Campbell Road\\Richardson, TX 75080}
\email{stephen.mckeown@utdallas.edu}
\urladdr{www.utdallas.edu/$\sim$sxm190098}

\keywords{Singular Yamabe problem, global conformal invariants, renormalized volume, extrinsic conformal geometry}
\subjclass[2020]{Primary: 53C40, 53C18; Secondary: 53A55, 58J90}

\begin{document}
\begin{abstract}
	Given an embedded closed submanifold $\Sigma^n$ in the closed Riemannian manifold $M^{n + k}$, where $k < n + 2$,
	we define extrinsic global conformal invariants of $\Sigma$ by renormalizing the volume
	associated to the unique singular Yamabe metric with singular set $\Sigma$. In case $n$ is odd, the renormalized volume is an absolute conformal invariant, while if $n$ is even, there is a conformally invariant
	energy term given by the integral of a local Riemannian submanifold invariant. In particular, the renormalized volume gives a global conformal invariant of a knot embedding in the three-sphere. We compute the variations
	of these quantities with respect to variations of the submanifold. We extend the construction of energies for even
	$n$ to general codimension by considering formal solutions to the singular Yamabe problem; except that, for each fixed $n$, there are finitely many $k \geq n + 2$, which we identify, for which the smoothness of
	the formal solution is obstructed and we obtain instead a pointwise conformal invariant. We compute the new quantities in several cases.
\end{abstract}

\maketitle

\input{tex/intro}
\input{tex/setup}
\input{tex/func}
\input{tex/smooth}
\input{tex/vol}
\input{tex/varsmooth}
\input{tex/var}
\input{tex/calcs}
\nocite{o70}
\bibliographystyle{amsalpha}
\bibliography{norm}
\end{document}

%% file: tex/intro.tex
\section{Introduction}
Suppose that $\Sigma^{n}$ is a closed embedded submanifold of the closed conformal manifold $(M^{n + k},[g])$, with $n - k + 2 > 0$. Recall that the \emph{singular Yamabe problem} asks for a conformal factor $u > 0$
so that the metric $u^{-2}g$ is complete on $M \setminus \Sigma$ and has constant scalar curvature.
In this paper, we use the unique solution of the singular Yamabe problem and the renormalized volume
construction to introduce (extrinsic) global conformal invariants of the pair $(M,\Sigma)$. Subject to the dimensional constraint, our quantities are always defined and conformally invariant; in case $n$ is odd, they are
global in $M$ and are new; in case $n$ is even, they are the integral over $\Sigma$ of an extrinsic Riemannian invariant, and are of similar type to the energies seen in \cite{gw99}.
The invariance of our volumes if $n$ is odd is rather surprising, as such behavior had not previously been noticed in the codimension-one case of the singular Yamabe problem; but see below. We also compute the derivative
of the variation of these quantities with respect to variations of $\Sigma$ in $M$. In case $n$ is even, the variation is the log-term coefficient in the expansion of $u$ itself, while if $n$ is odd, it is a conformally
invariant nonlocal term in the expansion.

The energy term when $n$ is even can be recovered in most cases even if $n - k + 2 \leq 0$. Although the solution to the singular Yamabe problem may no longer exist in that case, and it is not unique when it does,
we can still consider \emph{formal} solutions to the problem. These always exist, and they are smooth enough for our construction to work -- and give an invariant global energy --
except in a few exceptional codimensions for each dimension. In the latter cases, the formal expansion of the singular Yamabe function contains log terms that disrupt the construction of the energy; but their 
coefficients are themselves pointwise conformal invariants. In the critical case $k = n + 2$, both phenomena occur: there is an extrinsic pointwise conformal invariant of the correct weight so that its integral over $\Sigma$ is
a global invariant. The formal construction works in the case of odd $n$ as well, and renormalized volumes can be defined; but there is no longer a way to pick an invariant global singular Yamabe metric on $M$, so these
are essentially meaningless in terms of the geometry. In a few exceptional codimensions, pointwise extrinsic conformal invariants are obtained for odd $n$ as well.

The renormalized volume was introduced in conformal geometry in the setting of so-called Poincar\'{e}-Einstein (PE) metrics \cite{hs98,g99}, having been proposed by physicists
interested in the AdS/CFT correspondence \cite{mal98,wit98}. An asymptotically hyperbolic (AH) metric is a metric $g^+$ on the interior of a compact manifold
$X^{n + 1}$ with boundary $M^n = \partial X$ satisfying two conditions: if $\varphi \in C^{\infty}(X)$ is a defining function for $M$ -- that is, a function such that $M = \varphi^{-1}\left( \left\{ 0 \right\} \right)$
and $d\varphi|_{M}$ is nonvanishing -- then we require that (1) $\varphi^2g^+$ extends to a (say) $C^3$ metric on $X = \overline{X}$ and (2) $|d\varphi|_{\varphi^2g^+} \equiv 1$ along $M$. An AH metric is Poincar\'{e}-Einstein
if it additionally satisfies the Einstein condition $\ric(g^+) + ng^+ = 0$; or, more generally, $\ric(g^+) + ng^+ = O(\varphi^{\infty})$ (in which case, we might also call it a formally Poincar\'{e}-Einstein metric, if the
distinction is to be emphasized).
Now, among the defining functions, a particularly important class is the geodesic or special defining functions, which exist for any AH manifold. Let $[h] = [\varphi^2g|_{TM}]$ be the conformal class induced on $M$ by $g^+$
(called the \emph{conformal infinity}). By a result in \cite{gl91} (see \cite{g99}), for any $h \in [h]$ and $\varepsilon > 0$ small, there exists a unique diffeomorphism
$\psi:[0,\varepsilon)_r \times M \hookrightarrow X$ onto a neighborhood of $M$ in $X$ such that $\psi^*g^+ = \frac{dr^2 + g_r}{r^2}$, with $g_r$ a one-parameter family of metrics on $M$ and $g_0 = h$. For $g^+$ a
PE metric, it follows from the Einstein equations (see \cite{g99,cdls05}) that $g_r$ has an even expansion in $r$ up to order $n$:
\begin{equation}\label{eingexpeq}
	\begin{split}
		g_r = h + r^2g^{(2)} + r^{4}g^{(4)} +& \cdots + r^{n - 1}g^{(n - 1)}\\
		&+ r^{n}(\log r)K + r^ng^{(n)} + O(r^{n + 1}),
	\end{split}
\end{equation}
with each $g^{(i)}$ and $K$ a symmetric two-tensor on $M$. Moreover, if $n$ is odd, then $K = 0$ and $\tr_hg^{(n)} = 0$, while if $n$ is even, then $g^{(n - 1)} = 0$.

Now, the volume of $(X,g^+)$ is of course infinite, as $g$ blows up to second order at $M$. One may consider, however, the finite volume $\vol_{g^+}\left( \left\{ r > \varepsilon \right\} \right)$, where
$g^+$ is PE and $r$ a special defining function; and expanding this in $\varepsilon$, one finds (\cite{g99}) 
\begin{equation}\label{ahvolexpeq}
	\begin{split}
		\vol_{g^+}\left( \left\{ r > \varepsilon \right\} \right) = c_0\varepsilon^{-n} + c_1\varepsilon^{1 - n} +& \cdots + c_{n - 1}\varepsilon^{-1}\\
		&+ \mathcal{E}\log\left( \frac{1}{\varepsilon} \right) + V + o(1).
	\end{split}
\end{equation}
However, due to the parity properties of (\ref{eingexpeq}), one can show that $c_j = 0$ for all odd $j$. Moreover, if $n$ is odd, then $\mathcal{E} = 0$ and $V$ is a global conformal invariant in the sense that it does not
depend on which geodesic defining function is chosen. On the other hand, if $n$ is even, then $\mathcal{E}$ is a global conformal invariant in the same sense
and, in addition, is the integral over $M$ of a local Riemannian invariant. 
These invariants have been central objects of study in the theory of PE metrics, and even conformal geometry broadly, since their introduction \cite{g99,and01,cqy08,alexakis12,
gms12,gw99}. According to the Fefferman-Graham
holographic procedure \cite{fg85,fg12}, one can realize any conformal manifold $(M^n,[h])$ as the conformal infinity of a (formal) PE metric $(X^{n + 1},g^+)$, so the invariants can naturally be used to study $M$ and the relationship
between $M$ and $X$. In odd boundary dimensions, of course, $V$ will depend on the specific global metric $g^+$ chosen, and so is most interesting when a globally Einstein metric is available.

Barely after the introduction of renormalized volume, Graham and Witten introduced the concept of renormalized \emph{area} to study the extrinsic conformal geometry of submanifolds $\Sigma \subset M$ \cite{gw99}. 
Changing dimensional notation
from the prior paragraph to match our usage in the rest of the paper, suppose $\Sigma^n \subset M^{n + k}$ is a closed submanifold of codimension $k$ in the conformal manifold $M$. Let $M$ be realized as the conformal
infinity of a formal PE metric $(X^{n + k + 1},g^+)$. One can then (at least formally) realize $\Sigma$ as the boundary of a complete minimal submanifold $Y^{n + 1} \subset X$, which itself is AH with respect to the induced
metric.  One can renormalize the area of $Y$ with respect to the induced metric, in exactly the same way as the volume above. Once again, if $n$ is odd, the area is a global invariant of $Y$, while if $n$ is even,
a logarithmic term appears in the expansion of the area, which is conformally invariant and is the integral over $\Sigma$ of an \emph{extrinsic} Riemannian invariant term. In case $n = 2$, this is the Willmore invariant; thus,
the even-dimensional cases of the Graham-Witten result may be seen as generalizing the Willmore invariants. Once again, in the odd-dimensional case, the renormalized area depends on the global metric $g^+$ as well as the
global submanifold $Y$. For more recent work related to this construction, see \cite{am10,gr20,tyr22,cgktw24}.

The renormalized volume construction was significantly extended in \cite{g17,gw17} to the setting of singular Yamabe metrics. We here recall this theory. Suppose $(M^{n + 1},g)$ is a Riemannian metric on a manifold
$M$ with boundary $\Sigma^n$, and let $r$ be the distance function to $\Sigma$ in $M$.
The singular Yamabe (or Loewner-Nirenberg) problem asks whether there is a defining function $u$ for $\Sigma$ so that the complete metric $g^+ = u^{-2}g$ on $\mathring{M}$ has constant scalar curvature
$-n(n + 1)$. The problem was originally posed in \cite{ln74}, and by \cite{ln74,am88,acf92}, a solution always exists and is unique; moreover, by \cite{maz91sy}, it is polyhomogeneous, i.e., it is smooth in
$r$ and terms of the form $r^k(\log r)^j$. By uniqueness, it follows that $u$ is conformally invariant: if $\hat{g} = e^{2\omega}g$, then $\hat{u} = e^{\omega}u$. This, and the fact that $u$ is polyhomogeneous
and (thus) has a well-defined and locally determined expansion to order $n + 1$ at the boundary has led in recent years to the problem's heavy use as a means to study the extrinsic conformal geometry of the embedding
$\Sigma \hookrightarrow M$ \cite{gw15,gp18,jo22,cmy22}. 
Now, $g^+$ is actually an AH metric, and in \cite{g17,gw17}, it was realized that one could study the volume $\vol_{g^+}\left( \left\{ r > \varepsilon \right\} \right)$ and get an expansion
like (\ref{ahvolexpeq}). In this instance, no interesting parity properties of $u$ were observed, and none of the $c_i$ vanishes. Moreover, $V$ is generically not conformally invariant in any dimension: if the volume is
expanded with respect to the distance function $\hat{r}$ corresponding to $\hat{g}$, then $\widehat{V}$ may not equal $V$. However, the energy
$\mathcal{E}$ is always conformally invariant (and generically nonzero), and also gives the Willmore energy in case $n = 2$. There has been a great deal of study of these quantities in recent years
(\cite{gg19,gw19,cmy22,jo22}). We note that the paper \cite{gw17} also considers \emph{interior} hypersurfaces of a closed manifold, but
still focuses on the expansion of $u$ on one side of the hypersurface. 

The singular Yamabe problem has actually been posed, and solved, for singular sets of codimension higher than one. Suppose $\Sigma^n$ is a closed, embedded submanifold of the compact Riemannian manifold
$(M^{n + k},g)$; and suppose $n - k + 2 > 0$. Let $t$ be the distance function to $\Sigma$ with respect to $g$. (We use $t$ to emphasize that this distance, unlike $r$, is not smooth.)
By \cite{am88}, there always exists a unique $u \geq 0$ satisfying $u^{-1}\left( \left\{ 0 \right\} \right) = \Sigma$ such that the complete metric $g^+ = u^{-2}g$ has constant negative scalar curvature
$-(n - k + 2)(n + k - 1)$. 

Moreover, by \cite{maz91sy}, $u$ has a polyhomogeneous expansion in $t$. Recall that the normal exponential map of $\Sigma$ induces, for some neighborhood $U$ of $\Sigma$, a diffeomorphism
$U \setminus \Sigma \approx (0,\varepsilon) \times SN\Sigma$, where the latter factor is the normal sphere bundle over $\Sigma$. Then the result of \cite{maz91sy} may be stated as saying that $u$ has
an expansion of the form 
\begin{equation}\label{polyhom}
	\begin{split}
		u =  t + t^2u_2 +& \cdots + t^{n + 1}u_{n + 1} + t^{\gamma}u_{\gamma}\\
		&+ t^{n + 2}(\log t)\mathcal{L} + t^{n + 2}u_{n + 2} + o(t^{n + 2}),\\
	\end{split}
\end{equation}
with $\gamma \in (n + 1,n + 2)$ determined by $n$ and $k$, and with the $u_j$ and $\mathcal{L}$ smooth on $SN\Sigma$; and indeed, the expansion may be continued in powers of $t$ and $\log t$ to arbitrarily high order.

In this paper, we study the renormalized volume of higher-codimension singular Yamabe metrics $g^+$ by considering the expansion $\vol_{g^+}\left( \left\{ t > \varepsilon \right\} \right)$ in $\varepsilon$.
More specifically, we study the expansion
\begin{equation}\label{volexpintroeq}
	\begin{split}
		\vol_{g^+}\left( \left\{ t > \varepsilon \right\} \right) = c_0\varepsilon^{-n} + c_1\varepsilon^{1 - n} +& \cdots + c_{n - 1}\varepsilon^{-1}\\
		&+ \mathcal{E}_{n,k}\log\left( \frac{1}{\varepsilon} \right) + V_{n,k} + o(1).
	\end{split}
\end{equation}
The somewhat surprising result is that this problem manifests the same parity-dependence as in the Einstein case, as distinguished from the codimension-one singular Yamabe case on a manifold with boundary.
We will discuss this more below. We now state our main result. 
\begin{theorem}\label{nicecase}
	Suppose $(M^{n + k},[g])$ is a closed conformal manifold and $\Sigma^n$ is a closed, embedded submanifold, with $n - k + 2 > 0$. 
	Suppose $g^+ = u^{-2}g$ is the unique singular Yamabe metric satisfying the condition $R_{g^+} = (k - n - 2)(n + k - 1)$.
	
	Then in the expansion (\ref{volexpintroeq}), $c_j = 0$ for odd $j$. 
	In case $n$ is odd, $\mathcal{E}_{n,k} = 0$ and $V_{n,k}$ is a conformal invariant. In case $n$ is even, $\mathcal{E}_{n,k}$ is a conformal invariant, 
	and is the integral over $\Sigma$ of a local extrinsic Riemannian invariant.
\end{theorem}
The first question this theorem raises is -- why should it be true? To see the essence, recall that we may locally identify a deleted neighborhood $U$ of $\Sigma$ with the product $\Sigma \times \spheres^{k - 1} \times (0,\delta)_t$
via the normal exponential map.
Suppose $f$ is a smooth function on $M$. Then we may expand
\begin{equation*}
	f = f_0 + tf_1 + t^2f_2 + \cdots + t^nf_n + O(t^{n + 1}),
\end{equation*}
where each $f_i$ is a function on $\Sigma \times \spheres^{k - 1}$. The crucial point is that under the antipodal map $T:\spheres^{k - 1} \to \spheres^{k - 1}$, each $f_i$ transforms by
$T^*f_i = (-1)^if_i$. The same principle applies to the expansion of a smooth volume form in appropriate coordinates. So upon integration over $\spheres^{k - 1}$, the odd terms all vanish. (Compare \cite[section 9.3]{grayTubes}.)
There is much to be done before simply applying this intuition to the singular volume form $dV_{g^+}$ and obtaining the theorem. 
For example, a good deal of our paper is devoted to showing that $\frac{u}{t}$ is actually smooth to order
$o(t^{n})$, and not merely polyhomogeneous; and smoothness is required for this argument to work.
But the essential idea is that the parity-dependence that appeared before integration in the Poincar\'{e}-Einstein case
appears \emph{upon} integration in our case.

It was pointed out to us by Robin Graham that this behavior \emph{could} have been (though was not) noticed even in the codimension-one case, had the hypersurface been considered as two-sided and embedded
in the interior. Suppose $\Sigma^n$ is a hypersurface of a closed $M^{n + 1}$, rather than a boundary. If the singular Yamabe problem is solved (on each side, if $\Sigma$ is dividing), the same
parity behavior emerges if the coefficients are viewed as functions on the unit normal bundle $\spheres^0$. Alternatively, one could consider a \emph{signed} distance function $r$ and consider the volume
$\vol_{g^+}(\left\{ |r| > \varepsilon \right\})$.

Several interesting features of Theorem \ref{nicecase} bear remark. First, in the special case $\spheres^1 \hookrightarrow \spheres^3$, the renormalized volume $V_{1,2}$ 
gives a global conformal invariant of a knot embedding. This appears to be the same
kind of quantity as the writhe of a knot \cite{o03}. 
Unlike the writhe, our volume does not vanish on the equatorial unknot.

In general, the renormalized volume $V_{n,k}$ for $n$ odd will be a well-defined global invariant, due to uniqueness of the singular Yamabe metric.

Next, it is interesting to note the relationship between our quantity and the Graham-Witten renormalized area. In the case of $n$ even, this gives an energy of precisely the same sort as the Graham-Witten energy --
a global conformal invariant given by the integral of an extrinsic local invariant. Exactly what the relationship is bears further study; see section \ref{calcsec} for the analysis when $n = 2$. 
In the case of $n$ odd, it is likewise of the same type, but with an important difference:
the Graham-Witten renormalized area is well-defined only upon choosing a PE metric (and a minimal $Y$). Generally, only formal PE metrics may exist, and there is no known method to choose a best one. On the other hand, the volume
defined here is always canonical.

In fact -- finally -- we believe that in the case of odd $n$,
the renormalized volume defined here is the first such example of a renormalized volume that is a canonical conformal invariant associated to the geometric data. The volume in the
(codimension-one, boundary) singular Yamabe setting is well-defined, but is not generally conformally invariant; while the Einstein and Graham-Witten renormalized volumes and areas are conformally invariant, but generally
depend on a choice of formal Einstein metric (unless a genuine Einstein metric is available and is unique).

We also compute the variations of the invariant quantities $\mathcal{E}_{n,k}$ and $V_{n,k}$. To state the results, we note first that the quantity $\mathcal{L}$ in (\ref{polyhom}) is actually the restriction
to the unit normal bundle of a \emph{linear} function on the entire normal bundle, which we refer to as $\Lambda\mathcal{L}$. This may be interpreted as a section
of $T^*M|_{\Sigma}$, i.e., a one-form. Moreover, it is conformally invariant of weight $-n$. (By raising an index, we get a conformally invariant normal vector field.)
A more subtle phenomenon arises regarding the term $u_{n + 2}$. It is \emph{not} the restriction to the normal bundle of a linear function, but in case $n$ is odd, its
linear \emph{part}, which is not locally determined, determines a conformally invariant one-form $\Lambda u_{n + 2}$ on $N\Sigma$ by the formula
\begin{equation}
	\label{uform}
	(\Lambda u_{n+2})_p(X) = \frac{k}{\vol(\spheres^{k - 1})}\int_{SN_p\Sigma}u_{n + 2}(Y)\langle X,Y\rangle dV_{\bo}(Y).
\end{equation} 
Here, $\bo$ is the metric on the unit normal sphere $SN_p\Sigma$ at a point $p \in \Sigma$. When $k = 1$, the integral is a sum over two points. See Lemma \ref{oneformlem} and section \ref{smoothsysec} for details.
\begin{theorem}
	\label{varthm}
	Let $M,[g],\Sigma, n$, and $k$ be as in Theorem \ref{nicecase}. Write $h_0 = g|_{T\Sigma}$.
	If $n$ is odd and $k > 1$, then we have $\mathcal{L} = 0$. Suppose $\mathcal{F}:(-\delta,\delta) \times \Sigma \hookrightarrow M$ is a smooth variation
	of $\Sigma$ with $\mathcal{F}(0,\cdot) = \id_{\Sigma}$. Let $X = \frac{d}{ds}\mathcal{F}(s,\cdot)|_{s = 0} \in \Gamma(\Sigma,N\Sigma)$ be the variation field. For each $s$, let 
	$\Sigma_s = \mathcal{F}(s,\Sigma)$, and let $\mathcal{E}_{n,k}(s)$ and
	$V_{n,k}(s)$ be the energy and volume corresponding to $(M,\Sigma_s)$. Let
	\begin{equation*}
		C_{n,k} = (-1)^n\frac{(n + k)(n^2 + n + 2k - 4)}{k(n - k + 2)}\vol(\spheres^{k - 1}).
	\end{equation*}
	If $n$ is even, then
	\begin{equation}
		\label{evenvar}
		\left.\frac{d}{ds}\mathcal{E}_{n,k}(s)\right|_{s = 0} = C_{n,k}\int_{\Sigma}(\Lambda\mathcal{L})(X)dV_{h_{0}}.
	\end{equation}
	If $n$ is odd, then
	\begin{equation}
		\left.\frac{d}{ds}V_{n,k}(s)\right|_{s = 0} = C_{n,k}\int_{\Sigma}(\Lambda u_{n + 2})(X)dV_{h_0}.
	\end{equation}
\end{theorem}
This result for even $n$ bears a strong resemblance to the variation formula given in \cite{g17} for the variation of the energy (in all dimensions) for the boundary singular Yamabe case, and (as already discussed there) to the
variation of the energy term in the Einstein case, which is given by the Fefferman-Graham obstruction tensor that forms the first logarithmic term in the expansion of the Einstein metric -- a multiple of $K$ in (\ref{eingexpeq}).
See \cite{gh05}. In the case of $n$ odd, it is analogous to the variation formula of Anderson \cite{and01} for the variation of the renormalized volume of Einstein metrics, 
where the variation is the nonlocal term in the expansion of the Einstein metric. See also \cite{alb09}.

We comment that Theorem \ref{varthm} puts constraints on conformal Killing fields that are transverse to $\Sigma$.

When $n - k + 2 \leq 0$, the theory of the solution to the singular Yamabe problem 
\begin{equation}\label{syhighh}
	R_{u^{-2}g} = (k - n - 2)(n + k - 1)
\end{equation}
becomes quite complicated, and several of the appealing properties such as uniqueness
are lost. It is known that the prescribed scalar curvature in these codimensions must be positive or zero \cite{am88,maz91sy}. We pass over the critical case for the time being and consider the case $n - k + 2 < 0$.  
Solutions to (\ref{syhighh}), when they exist, are spectacularly non-unique. Moreover, there exist many solutions that are not polyhomogeneous \cite{maz91sy}.
The most complete existence result to date is \cite{mp96}, which says that if $(M,[g])$ has nonnegative Yamabe constant, then there exists an infinite-dimensional family of polyhomogeneous solutions. Now,
in this case, the expansion generically includes terms of the form $t^{\gamma}u_{\gamma}$ for $\gamma \in (0,n+1) \setminus \mathbb{N}$.
These would alter the volume expansion (\ref{volexpintroeq}) by adding terms of non-integral powers in $\varepsilon$. With that modification, Theorem \ref{nicecase} would still otherwise hold
(including invariance of volume for odd $n$). However, due to the non-uniqueness of solutions, the volume would not be an interesting invariant of $(M,\Sigma,[g])$ unless one could choose a conformally invariant
condition to uniquely specify an existent singular Yamabe metric from all the possibilities. Given this,
we here consider instead \emph{formal} solutions $u$ having the expansion 
\begin{equation}\label{formaleq}
	u =  t + t^2u_2 + \cdots + t^{n + 1}u_{n + 1} + t^{n + 2}(\log t)\mathcal{L} + t^{n + 2}u_{n + 2} + o(t^{n + 2})
\end{equation}
and satisfying $R_{g^+}  = (k - n -  2)(n + k - 1) + O(t^{n + 2})$ (where, indeed, the expansion may be continued
in powers of $t$ and $\log t$ until the order power on the right-hand side is made as large as desired or even infinite). This works generically, but for some exceptional $k$, logarithmic terms appear before order
$t^{n + 2}$ and disrupt the construction.

We wish, then, to study the volume expansion (\ref{volexpintroeq}) for formal solutions (\ref{formaleq}) of the singular Yamabe equation.
As we will see, for even $n$ this will allow us to extend the existence
of the conformally invariant energies to most codimensions $k$.
We must be quite careful, however, because although such formal expansions exist generally, as already mentioned there are various values of $n$ and $k$ for which the ``indicial root'' (see section \ref{smoothsec}) becomes integral,
and logarithmic powers are introduced into the expansion at low orders, which complicates things significantly. Actually, and interestingly, when $n$ is even, half of the integral indicial roots do \emph{not} lead to logarithmic terms.
This may be seen as analogous to the PE setting for odd boundary dimension, where an indicial root leads to a loss of uniqueness, but not to a logarithmic term. When we do get a logarithm, we will not be able to run our volume
expansion to obtain an energy; but our recompense will be that the log coefficient is a pointwise conformal invariant. 

We summarize as briefly as possible the situation for formal expansions when $n - k + 2 \leq 0$. The parity of the order at which the formal expansion becomes undetermined turns out to be
decisive. Fix $n \geq 2$. Set $P_n(p) = n + 2 + 2np - 4p^2$ and $Q_n(p) = 2n + 1 + 2(n-2)p - 4p^2$. If $n$ is even, define
\begin{equation}\label{eveneo}
	\begin{split}
		E_n &= \left\{ P_n(p): p = 0,\cdots,\left\lfloor \frac{n}{4}\right\rfloor  \right\}\\
		O_n &= \left\{ Q_n(p) : p = 0,\cdots,\left\lfloor\frac{n}{4}\right\rfloor\right\}.
	\end{split}
\end{equation}
If $n$ is odd, define
\begin{equation}\label{oddeo}
	\begin{split}
		E_n &= \left\{ P_n(p): p = 1,\cdots,\left\lfloor \frac{n}{2}\right\rfloor \right\}\\
		O_n &= E_n \cup \left\{ n + 2 \right\}.
	\end{split}
\end{equation}
Following is a summary of the existence situation for formal expansions.
\begin{theorem}\label{expandthm}
	Let $n \geq 2$.
	\begin{enumerate}[(a)]
		\item\label{partfirst} Suppose that $k \notin (E_n \cup O_n)$. There exists a function 
			$u$ having the expansion (\ref{formaleq}) and satisfying the scalar curvature condition $R_{u^{-2}g} = (k - n - 2)(n + k - 1) + O(t^{n + 2})$.
			Such $u$ is unique mod $O(t^{n + 2})$. If $n$ is odd and $k > 1$, then $\mathcal{L} = 0$. In this case, the solution may be made unique mod $o(t^{n + 2})$ by imposing the conformally invariant
			condition that (\ref{uform}) vanish for all $p \in \Sigma$ and all $X \in N_p\Sigma$. If $n$ is even or $k = 1$, $\mathcal{L}$ is a pointwise conformal invariant of weight $-(n + 1)$.
		\item\label{partodd} Suppose $k \in O_n$. There exists $u$ having the expansion (\ref{formaleq}) and satisfying $R_{u^{-2}g} = (k - n - 2)(n + k - 1) + O(t^{n + 2})$.
		Such $u$ is unique mod $O(t^{n + 2})$ given the additional constraint that, for each positive 
		root $\nu$ to $Q_n(\nu) = k$, the coefficient $u_{\nu + 1}$ integrates to zero on
		each fiber of the normal sphere bundle to $\Sigma$; this is equivalent to requiring that $\frac{u}{t}$ be smooth on $M$ mod $o(t^n)$, and is thus a conformally invariant condition.
		If $n$ is odd and $k > 1$, then $\mathcal{L} = 0$, and the expansion may be made unique mod $o(t^{n + 2})$ by imposing the same conformally invariant condition as in \ref{partfirst}. If $n$ is even or $k = 1$, then
		$\mathcal{L}$ is a pointwise conformal invariant of weight $-(n + 1)$.
	\item\label{parteven} Suppose that $k \in E_n$. Let $\nu$ be the smallest positive root to $P_n(\nu) = k$.
			\begin{enumerate}[(i)]
				\item\label{kennotn2} If $\nu \neq \frac{n}{2}$, a formal solution $u$ exists of the form
	\begin{equation*}
		u = t + t^2u_2 + \cdots + t^{\nu}u_{\nu} + t^{\nu + 1}(\log t)A + t^{\nu + 1}u_{\nu + 1},
	\end{equation*}
	satisfying $R_{u^{-2}g} = (k - n - 2)(n + k - 1) + o(t^{\nu + 1})$. Such $u$ is unique to order $o(t^{\nu + 1})$ 
	if subjected to the requirement that the integral of $u_{\nu + 1}$ over each fiber of the normal sphere bundle to $\Sigma$ vanish; however, this is not a conformally invariant condition.
				\item\label{kenn2} If $\nu = \frac{n}{2}$, then $n \in 4\mathbb{N}$ and $k = \frac{n^2}{4} + n + 2$.  In this case, a formal solution $u$ exists of the form
	\begin{equation*}
		u = t + t^2u_2 + \cdots + t^{\frac{n}{2}}u_{\frac{n}{2}} + t^{\frac{n}{2} + 1}(\log t)^2A + t^{\frac{n}{2} + 1}u_{\frac{n}{2} + 1},
	\end{equation*}
	satisfying $R_{u^{-2}g} = (k - n - 2)(n + k - 1) + o(t^{\frac{n}{2} + 1})$. Uniqueness is the same as in the preceding case.
			\end{enumerate}
	Finally, in either of cases \ref{kennotn2} or \ref{kenn2},
	$A$ is a pointwise conformal invariant: if $\hat{g} = e^{2\omega}g$, then $\widehat{A} = e^{-\nu\omega}A$.
	\end{enumerate}
\end{theorem}
As discussed above and within, the term $\mathcal{L}$ when $n$ is even may be seen as defining a conformally invariant one-form on the normal bundle -- or, by raising an index, a conformally invariant normal vector field
along $\Sigma$. We also note that the critical case $n - k + 2 = 0$ (or $k = n + 2$) for $n$ even will always fall in part \ref{parteven}\ref{kennotn2} with $\nu = n$, and will thus
yield a pointwise conformal invariant of weight $-n$. Its integral will thus be a global conformal invariant.

We may now state our most general theorem on volume expansions; Theorem \ref{nicecase} is merely a corollary.
\begin{theorem}\label{mainthm}
	Suppose $(M^{n + k},[g])$ is a closed conformal manifold and $\Sigma^n$ is a closed, embedded submanifold, with $k \notin E_n$.
	Suppose that $g^+ = u^{-2}g$ is a formal singular Yamabe metric satisfying the condition $R_{g^+} = (k - n - 2)(n + k - 1) + O(t^{n + 1})$, and that $u$ satisfies (\ref{formaleq}). If $k \in O_n$, assume also that
	$u$ satisfies the smoothness conditions given in Theorem \ref{expandthm}\ref{partodd}.

	Then in the expansion (\ref{volexpintroeq}), $c_j = 0$ for odd $j$.
	Suppose $g, \hat{g} \in [g]$ are smooth metrics with distance functions $t$ and $\hat{t}$, respectively, and that the expansion (\ref{volexpintroeq}) is computed with respect to each.
	If $n$ is odd, then $\widehat{\mathcal{E}}_{n,k} = \mathcal{E}_{n,k} = 0$ and $\widehat{V}_{n,k} = V_{n,k}$. If $n$ is even, then $\widehat{\mathcal{E}}_{n,k} = \mathcal{E}_{n,k}$. In the latter case,
	$\mathcal{E}_{n,k}$ is the integral over $\Sigma$ of a local extrinsic Riemannian invariant, and is also independent of the formal singular Yamabe metric $g^+$ conformal to $g$.
\end{theorem}
We remark again that allowing irrational powers higher than one into the expansion (\ref{formaleq}) would not change any of the claims in this theorem, except that (\ref{volexpintroeq}) would contain non-integral powers
of $\varepsilon$ in addition to those present. Thus, this theory can be applied to the polyhomogeneous global solutions produced in \cite{mp96}, except for the exceptional codimensions.

Because of the lack of uniqueness of singular Yamabe solutions, the variation of renormalized volume (when it is defined) is not well-posed in general. However, because the energy is the integral of a local extrinsic quantity,
it still makes sense to discussion the derivative of energy when $\Sigma$ varies for even $n$; and in that case, (\ref{evenvar}) still holds.

A natural question -- studied recently for the Graham-Witten setting \cite{cgk23} -- is whether a $Q$-curvature and associated GJMS-type operator exist for the energies $\mathcal{E}_{n,k}$ in the case of $n$ even.
We are preparing a paper that answers this in the affirmative.

It would be interesting to understand better the conformal knot invariant $V_{1,2}$ when $M = \spheres^3$. We compute it in section \ref{calcsec} in the case of the equatorial unknot, obtaining the value $V_{1,2} = -4\pi^2$. Computing it for
other conformal embeddings of the unknot or for other knots would help shed light on its properties. It would likewise be interesting to know its relationship to other conformal energies.

The paper is organized as follows. In section \ref{setupsec}, we fix notations and coordinate systems. In section \ref{funcsec},
we introduce the core ideas that we will use to analyze smoothness. This is a matter of paying careful attention to the expansion of a function
in spherical harmonics. In section \ref{smoothsec}, we apply this theory to analyze several functions. Theorem \ref{usmooththm} is the heart of the paper. Not only does it develop the formal expansions
of $\frac{u}{t}$ and thus yield Theorem \ref{expandthm}, but more importantly, it shows that these expansions constitute functions that are smooth on $M$ and thus have the parity properties we need.
For negative curvature, the actual singular Yamabe function $u$ is therefore $t$ times a function smooth up to order $o(t^n)$.
We carry out a similar analysis for the smoothness of the ratio of the distance functions for conformally related metrics. We prove Theorem \ref{mainthm} in section \ref{volsec}, wherein the preceding sections allow us to use
much the same argument as in \cite{g99}. In section \ref{varsmoothsec}, we discuss how various quantities change under variations of $\Sigma$, in preparation for proving Theorem \ref{varthm}. We prove this theorem in section
\ref{varsec}. Finally, in section \ref{calcsec}, we compute the quantities we have defined in the case of the equatorial subspheres of any sphere and for the case $n = 2$.

\textbf{Acknowledgments.} The authors gratefully acknowledge Andrea Malchiodi for suggesting the problem to them. They are grateful as well for helpful conversations with Robin Graham, Rafe Mazzeo, Jeffrey Case,
Mieczyslaw Dabkowski, and Malcolm Hoffman. The second author was partially supported by Simons grant \#966614.

%% file: tex/setup.tex
\section{Preliminaries}\label{setupsec}
Throughout, we let $M^{n + k}$ be a closed Riemannian $(n+k)$-manifold with metric $g$, and we let $\Sigma^n$ be an embedded closed submanifold of codimension $k$. 
We write $h_0 = g|_{T\Sigma}$.

We recall the existence of Fermi coordinates for $\Sigma \hookrightarrow M$; see \cite{leeIRM,grayTubes}. These are coordinates $\left\{ x^1,\dots,x^{n + k} \right\}$ in a neighborhood $\mathcal{U}$ in $M$
of any point $p \in \Sigma$
such that the following properties hold. 
\begin{enumerate}[(i)]
	\item $\Sigma \cap \mathcal{U} = \left\{ q: x^{n + 1} = \dots = x^{n + k} = 0 \right\}$; 
	\item for any $(x^1,\cdots,x^n,x^{n + 1},\cdots,x^{n + k})$ in the range of the coordinates, the curve $s \mapsto (x^1,\cdots,x^n,sx^{n + 1},\cdots,sx^{n + k})$ describes a geodesic;
	\item the distance $t$ to $\Sigma$ may be written in these coordinates as 
		\begin{equation}\label{distform}
			t = \sqrt{(x^{n + 1})^2 + \dots + (x^{n + k})^2};
		\end{equation}
	\item along $\Sigma$, $\partial_{i} \perp \partial_a$ for $1 \leq i \leq n$ and $n + 1 \leq a \leq n + k$;
	\item\label{fermfour} along $\Sigma$, $\partial_a \perp \partial_b$ for $n + 1 \leq a, b \leq n + k$ and $a \neq b$;
	\item along $\Sigma$, $\partial_ag_{bc} = 0$ for $n + 1 \leq a,b,c \leq n + k$.
\end{enumerate}
These coordinates induce an identification, for some $V^k \subseteq \mathbb{R}^k$ containing the origin,
\begin{equation}
	\mathcal{U} \approx (\Sigma^n \cap \mathcal{U}) \times V^k \subseteq (\Sigma^n \cap \mathcal{U}) \times \mathbb{R}^{k}. \label{fermident}
\end{equation}
We use these coordinates throughout the paper, and we introduce the following conventions. We use the Roman indices $1 \leq i,j,k \leq n$ for the coordinates tangent to the $\Sigma$ factor in the just-mentioned
identification, while we use the Roman indices $n + 1 \leq a,b,c \leq n + k$ for the coordinates in the normal directions. 
We will also name the latter coordinates $y^a = x^a$ ($n + 1 \leq a \leq n + k$) when dealing especially with them,
to emphasize that they are normal coordinates. Thus, we might write a point as $(x^i,y^a)$. We will use the capital Roman indicies $1 \leq A,B,C \leq n + k$ to run over all the coordinates. 

We next introduce cylindrical, or polar Fermi, coordinates, which are defined in terms of the Fermi coordinates and will be equally useful to us. Any point $(y^{n + 1},\dots,y^{n + k})$ in the second factor of
(\ref{fermident}) may be written $t\omega$, where $t \in [0,\infty)$ and $\omega \in \spheres^{k - 1}$. We may thus write the point as $(\omega,t)$, and this induces an identification 
\begin{equation}\label{fermpolident}
	\mathcal{U} \setminus \Sigma \approx (\Sigma \cap \mathcal{U}) \times \spheres^{k - 1} \times (0,\delta)
\end{equation}
(upon restricting $\mathcal{U}$ if necessary). Since $\Sigma$ is compact, we may cover it by finitely many such charts,
and by restricting if necessary, may assume that all of them lie in (and cover) a single tubular neighborhood $\mathcal{T}$ of $\Sigma$, say of radius $\delta$.
In these coordinates, we will identify a point by the triple $(p,\omega,t)$, with $t$ the distance to $\Sigma$. We may choose coordinates $\omega^{\mu}$ locally on $\spheres^{k - 1}$, $n + 1 \leq \mu \leq n + k - 1$, so that
$(x^i,\omega^{\mu},t)$ are coordinates on $\mathcal{U} \setminus \Sigma$. When using these coordinates, we use the convention $1 \leq I,J,K \leq n + k$ (rather than $A,B,C$) to make it clear we are using cylindrical
coordinates; we also use $z^I$ to run over all the coordinates, with $z^i := x^i$, $z^{\mu} := \omega^{\mu}$, and
$z^{n + k} := t$.

It is inconvenient that $t$ is not smooth at $\Sigma$, and that these coordinates do not extend to all of $\mathcal{U}$. This, of course, is just the same problem presented by polar coordinates at the origin; and the solution,
as usual, is to pass to the \emph{blow-up space} $\widetilde{M}$, as follows.
Let $N\Sigma=(T\Sigma)^\perp \subseteq TM$ be the normal bundle of $\Sigma$ in $(M,g)$. Write $\St =SN\Sigma$, where $SN\Sigma \subset N\Sigma$ is the unit 
normal sphere bundle of $\Sigma$. 
Thus, $\St$ is the total space of a fibration over $\Sigma$ with fiber $\spheres^{k-1}$. Define the blow-up of $M$ along $\Sigma$ as $\Mt = (M\backslash \Sigma) \sqcup \St$ and the blow-down map $\beta:\Mt\rightarrow M$ by
$\beta(p)=p$ for $p\in M\backslash \Sigma$ and $\beta(\widetilde{p}) = \pi(\widetilde{p})$ for $\widetilde{p}\in\St$, where $\pi$ is the projection map. 
Then $\Mt$ has a unique smooth structure such that the pushforward $d\beta_{\widetilde{p}}$ is a local diffeomorphism away from $\St$ and has rank $n+1$ for each $\widetilde{p}\in\St$, and such that
the restriction of each chart to the fibers is compatible with the smooth structure there; see
\cite{mel08}. Note that the blowup $\widetilde{M}$ is a compact manifold \emph{with boundary} $\partial\widetilde{M} = \widetilde{\Sigma}$.

Suppose $(\mathcal{U},\varphi)$ is a Fermi chart about some $q\in \Sigma$. 
Under $\beta $ the functions $x^i$, $\omega^\mu$ and $t$, defined as cylindrical coordinates above,
lift to smooth coordinates on a neighbourhood $\widetilde{\mathcal{U}} = \beta^*\mathcal{U}$ of $\St$ in $\Mt$:
\begin{equation*}
	\widetilde{\mathcal{U}} = \beta^{-1}(\mathcal{U}) \approx (\Sigma \cap \mathcal{U}) \times \spheres^{k - 1} \times [0,\delta).
\end{equation*}
In these coordinates, the metric $\beta^*g$ on $\Mt$ in a neighborhood of $\St$ can be written as 
\begin{equation}\label{gpolar}
	\beta^*g =\begin{pmatrix}
    h&t^2a&0\\t^2a^T&t^2b&0\\0&0&1
\end{pmatrix},
\end{equation}
where $h\in \Gamma\big(\widetilde{\mathcal{U}}, \sym^2(T^*\Sigma)\big)$ and $b\in \Gamma \big(\widetilde{\mathcal{U}},\sym^2(T^*(\spheres^{k-1}))\big)$ are symmetric two-tensors and $a\in \Gamma\big(\widetilde{\mathcal{U}},
T^*\Sigma \otimes T^*(\spheres^{k-1})\big)$. 
Using Schur complements, we see that its inverse is
\[\beta^* g^{-1} =\begin{pmatrix}
    h^{-1}& \mf{a}  &0\\
    \mf{a}^T&\frac{1}{t^2}b^{-1} + \mf{b} &0\\
    0&0&1
\end{pmatrix} + C,\]
where $\mf{a}=-h^{-1}ab^{-1}$, $\mf{b} = b^{-1}a^T\mf{a}$, and $C\in\Gamma(\widetilde{\mathcal{U}},\sym^2(T^*M))$ is such that $C^{AB} = O(t^2)$ for all $A,B$ and $C^{tJ}=0$ for all $J$.
Writing $\alpha = b-t^2a^Th^{-1}a$, we have $\det \beta^*g = t^{2(k-1)}\det h \det \alpha$.

We note the following useful fact, which follows easily from the properties of Fermi coordinates and the change-of-variable formula for metrics.
\begin{lemma}
	At $t = 0$, the metric $b$ restricted to each fiber is isometric to the round metric $\bo$ on the sphere $\spheres^{k - 1}$. That is, $b|_{t = 0} = \bo$.
\end{lemma}
Note that different choices of Fermi coordinate systems are related to each other -- at least in the $y^a$ -- by elements of the orthogonal group $O(k)$. Since the sphere metric is invariant with respect to the orthogonal group,
it follows that the round metric $\bo$ on each fiber $SN\Sigma$ is induced in a well-defined way, independent of any choice of coordinates.

\noindent
\textbf{Conventions.} The Riemann curvature tensor is defined so that in local coordinates, the Ricci tensor is given by $R_{AB} = g^{CD}R_{ACDB}$. By smooth, we mean $C^{\infty}$. The Laplacian is given by
$\Delta_g = \diver_g\grad_g$. The second fundamental form is the normal part of the ambient connection restricted to vectors tangent to $\Sigma$. The mean curvature is the trace of the second fundamental form.

%% file: tex/func.tex
\section{Functions of Polynomial Type}\label{funcsec}
In this section, we develop the tools we will use to analyze the smoothness on $M$ of functions smooth on $\widetilde{M}$. If a function $f: \mathbb{R}^n \to \mathbb{R}$ is smooth on
$\mathbb{R}^n \setminus \mathbb{R}^k$ and is homogeneous of degree $j \geq 0$ in distance to $\mathbb{R}^k$, it is smooth on $\mathbb{R}^n$ if and only if it is a polynomial. We wish to formalize
this fact in general, in terms of polar coordinates and using spherical harmonics.

As an example of the ideas, we consider some functions written in polar coordinates on
$\mathbb{R}^2$. The function $t\cos(\theta)$ is obviously smooth, since it is just $x$; and similarly $x^2y = t^3\cos^2(\theta)\sin(\theta) = \frac{1}{4}t^3(\sin(3\theta) + \sin(\theta))$. On the other hand,
$t^2\cos^4(\theta) = \frac{1}{8}t^2(3 + 4\cos(2\theta) + \cos(4\theta))$ is not smooth, since it is just $x^2\cos^2(\theta)$. The source of the problem is that it has higher powers -- or, equivalently, higher frequencies -- 
of sines and cosines than the power of $t$ (which is to say, it is not a polynomial). Another function that fails to be smooth is $t^2\sin(\theta) = ty$. Here, the problem is a parity 
mismatch between the power of $t$ and the trigonometric power, so a (non-smooth) factor of $t$ is left over.

We begin by recalling some facts about spherical harmonics. Unless otherwise stated, throughout this section assume $k > 1$. 
A function $Y \in C^{\infty}(\spheres^{k - 1})$ is a \emph{spherical harmonic} if it is an eigenfunction of the spherical Laplacian $\Delta_{\bo}$. Since it is known
that the eigenvalues of $\Delta_{\bo}$ are $-j(j + k - 2)$, where $j \in \mathbb{N} \cup \left\{ 0 \right\} =: \mathbb{N}_0$, we can classify the spherical harmonics by $j$; thus, we set
\begin{equation*}
	\mathcal{H}_j(\spheres^{k - 1}) = \left\{ u \in C^{\infty}(\spheres^{k - 1}): \Delta_{\bo}u + j(j + k - 2)u = 0 \right\}.
\end{equation*}
The elements of $\mathcal{H}_j$ are called the spherical harmonics of degree $j$.
We next prove the following; it is standard in three dimensions, but we did not immediately find a statement in general dimension.
\begin{proposition}\label{contractprop}
	Suppose that $Y_p \in \mathcal{H}_{p}(\spheres^{k - 1}), Y_q \in \mathcal{H}_q(\spheres^{k - 1})$. Then we may write $Y_pY_q = \sum_{l = 0}^{p + q}Z_l$, with $Z_l \in
	\mathcal{H}_l(\spheres^{k - 1})$. Moreover, $Z_l = 0$ if $p + q - l \equiv 1$ (mod $2$).
\end{proposition}
\begin{proof}
	It is a standard fact that if $Y \in \mathcal{H}_{p}$, then it is the restriction to $\spheres^{k - 1}$ of a homogeneous harmonic polynomial of degree $p$ on $\mathbb{R}^k$; for this and other standard properties, see e.g.
	\cite{sw71}. Suppose, then, that $Y_p$ is the restriction of $P_p$ and $Y_q$ is the restriction of $P_q$, these each being homogeneous harmonic polynomials on $\mathbb{R}^k$. Clearly enough, then,
	$Y_pY_q$ is the restriction of $P_pP_q$, which is a polynomial of degree $p + q$ on $\mathbb{R}^k$, homogeneous of degree $p + q$ (but generally not harmonic).

	Now, by Theorem 2.1 of \cite{sw71}, any homogeneous polynomial on $\mathbb{R}^k$, such as $P_pP_q$, may be written in the form
	\begin{equation*}
		P_p(x)P_q(x) = Q_0(x) + |x|^2Q_1(x) + \cdots + |x|^{2l}Q_l(x),
	\end{equation*}
	where each $Q_j(x)$ is a homogeneous harmonic polynomial of degree $p + q - 2j$, for $0 \leq j \leq l \leq \frac{p + q}{2}$.
	Restricting this to $\spheres^{k - 1}$ exhibits $Y_pY_q$ as a sum of spherical harmonics of constant parity and of degree bounded by $p + q$, as claimed.
\end{proof}

We wish to use these facts to analyze functions on $\widetilde{M}$. We define a function space on $\widetilde{\Sigma}$ by
\begin{equation*}
	\mathcal{H}_j(\widetilde{\Sigma}) := \left\{ u \in C^{\infty}(\widetilde{\Sigma}): \Delta_{\bo}u + j(j + k - 2)u = 0 \right\}.
\end{equation*}
Here, once again, $\bo$ is the round metric induced on each fiber by $g$. We recall that the eigenvalues of the round metric are precisely
$\left\{ -j(j + k - 2): j \in \mathbb{N} \cup \left\{ 0 \right\} \right\}$. We similarly define $\mathcal{H}_j(\widetilde{U})$ for any open subset of $\widetilde{\Sigma}$.

We define $\Xi_j$ to be the set $\left\{ (p,u): p \in \Sigma \text{ and } u \in \mathcal{H}_j(SN_p\Sigma) \right\}$, where $\mathcal{H}_j(SN_p\Sigma)$ is the isomorphic copy of $H_j(\spheres^{k - 1})$ for which the domain
of the functions is $SN_p\Sigma$. Then $\Xi_j$ is a finite-rank vector vector bundle over $\Sigma$ for each $j$. Clearly $\mathcal{H}_j(\widetilde{\Sigma})$ is nothing but the set of its sections:
$\mathcal{H}_j(\widetilde{\Sigma}) = \Gamma(\Xi_j)$.

The following will be useful facts about these spaces.
\begin{proposition}\label{hsigmaprop}
	\begin{enumerate}[(a)]
		\item\label{hsigmapropparttens} Let $p \in \Sigma$. There is a neighborhood $U$of $p$ in $\Sigma$ such that
			$\mathcal{H}_j(\beta^{-1}(U))$ is isomorphic to $C^{\infty}(U) \otimes \mathcal{H}_j(\spheres^{k - 1})$.
		\item\label{hsigmaproppartoplus} The algebraic direct sum $\oplus_{j = 0}^{\infty}\mathcal{H}_j(\widetilde{\Sigma})$ is an associative algebra.
	\end{enumerate}
\end{proposition}
\begin{proof}
	\ref{hsigmapropparttens} The bundle $\Xi_j$ is locally trivializable, so this follows from the identity $\mathcal{H}_j(\widetilde{\Sigma}) = \Gamma(\Xi_j)$.

	\ref{hsigmaproppartoplus} We may cover $\Sigma$ by neighborhoods $U_1,\cdots,U_N$ as given in \ref{hsigmapropparttens}, and choose a subordinate partition of unity
	$\psi_1,\cdots,\psi_N$. These induce a partition of unity on $\widetilde{\Sigma}$, which we refer to by $\tilde{\psi}_1,\cdots,\tilde{\psi}_N$. Suppose $u \in \mathcal{H}_j(\widetilde{\Sigma})$. We may write
	$u = \sum_{i = 1}^N\tilde{\psi}_iu$, and because each $\psi_i$ is constant in the fiber coordinates, $\psi_iu \in \mathcal{H}_j(\beta^{-1}(U_i))$. The result now follows from \ref{hsigmapropparttens} and
	Proposition \ref{contractprop}.
\end{proof}
Notice that there is a canonical identification $\mathcal{H}_0(\widetilde{\Sigma}) \approx C^{\infty}(\Sigma)$.

We next set
\begin{equation*}
	\mathcal{Y}_j(\widetilde{\Sigma}) = \begin{cases}
		\mathcal{H}_0(\widetilde{\Sigma}) \oplus \mathcal{H}_2(\widetilde{\Sigma})
		\oplus \cdots \oplus \mathcal{H}_j(\widetilde{\Sigma}) & j \text { even}\\
		\mathcal{H}_1(\widetilde{\Sigma}) \oplus \mathcal{H}_3(\widetilde{\Sigma}) \oplus
		\cdots \oplus \mathcal{H}_{j}(\widetilde{\Sigma}) & j \text{ odd}.\end{cases}
\end{equation*}
For $0 \leq m \leq j$, we let $\pi_m:\mathcal{Y}_j(\widetilde{\Sigma}) \to \mathcal{H}_m(\widetilde{\Sigma})$
be the natural projection if $j \equiv m $ (mod $2$), and $0$ otherwise. We will use the following properties of these spaces.

\begin{proposition}\label{hyprop}
	\begin{enumerate}[(a)]
		\item\label{hypropprod} For $Y_i \in \mathcal{Y}_i(\widetilde{\Sigma})$ and $Y_j \in \mathcal{Y}_j(\widetilde{\Sigma})$, we have $Y_iY_j \in \mathcal{Y}_{i + j}(\widetilde{\Sigma})$. The analogue holds
			for any finite product.
		\item\label{hypropsmooth} If $Y \in \mathcal{Y}_i(\widetilde{\Sigma})$, then $t^iY \in \beta^*C^{\infty}(\mathcal{\mathcal{T}})$.
		\item\label{hypropdiff} For $f \in C^{\infty}(M)$, $\frac{\partial^j}{\partial t^j}(f\circ\beta)|_{t = 0} \in \mathcal{Y}_j(\widetilde{\Sigma})$.
		\item\label{hypropsolve} If $F \in \mathcal{Y}_i(\widetilde{\Sigma})$ and $-\mu$ is not an eigenvalue of $\Delta_{\mathring{b}}$, then the unique smooth solution $w$ to
			$\Delta_{\mathring{b}}w + \mu w = F$ is also in $\mathcal{Y}_i(\widetilde{\Sigma})$. In particular, this holds if $\mu < 0$.

			If $-\mu = \lambda_p := -p(p + k - 2)$ is an eigenvalue, but $p$ is the opposite parity from $i$, then there still exists a unique solution $w$ satisfying $w \in \mathcal{Y}_i(\widetilde{\Sigma})$.
	\end{enumerate}
\end{proposition}
\begin{proof}
	\ref{hypropprod} This follows from Proposition \ref{hsigmaprop}\ref{hsigmaproppartoplus}.

	\ref{hypropsmooth} We first show this for $Y = H_i \in \mathcal{H}_i(\widetilde{\Sigma})$. Moreover, we will show it locally, which is of course enough. So by Proposition \ref{hsigmaprop}\ref{hsigmapropparttens}, 
	there is a neighborhood $U$ in $\Sigma$ such that, on $\beta^{-1}(U)$, we may
	consider $H_i(x,\omega) = f(x)H(\omega)$, where $f \in \beta^*C^{\infty}(\Sigma)$ and $H \in \mathcal{H}_i(\spheres^{k - 1})$. Now, by (\ref{distform}), the distance $t$ takes the same form in Fermi coordinates as in
	$\mathbb{R}^{k}$; and so $t^iH(\omega)$ defines a homogeneous harmonic polynomial in the coordinates $x^{n + 1},\cdots,x^{n + k}$ (or rather their pullbacks by $\beta$). Thus, $t^if(x)H(\omega)$ defines a smooth
	function on $U$.

	For $H_{i - 2j} \in \mathcal{H}_{i - 2j}(\widetilde{\Sigma})$, we have $t^iH_{i - 2j} = t^{2j}(t^{i - 2j}H_{i - 2j})$. The same argument as above works for the parenthetical factor, while
	$t^{2j}$ is already a smooth function on $M$.

	\ref{hypropdiff} We prove this locally in Fermi coordinates, again by partition of unity. Without loss of generality, then, we suppose $f$ is supported in a Fermi coordinate patch $U$.
	Suppose we write our coordinate system as $(x^1,\cdots,x^n,y^{n + 1},\cdots,y^{n + k})$. We write 
	\begin{equation}
		\label{ca}
		c_a = \frac{y^a}{t} = \partial_at 
	\end{equation}
	on the complement of $\Sigma$; the pullbacks of the $c_a$ by $\beta$ extend to smooth functions on
	$\widetilde{\Sigma}$. Now, since $\partial_t = \grad_{g}(t)$, we have
	\begin{equation}\label{delt}
		\frac{\partial}{\partial t} = c_ag^{ab}\frac{\partial}{\partial y^b}.
	\end{equation}
	By definition of the spherical coordinates $\omega^{\mu}$, the vector field $\partial_t$ commutes with $c_a$. Therefore, for a smooth function $f:U \to \mathbb{R}$, we have (away from $\Sigma$)
	\begin{align*}
		\partial_tf &= \sum_{a,b = n + 1}^{n + k}c_ag^{ab} \partial_bf\\
		\partial_t^2f &= \sum_{a_1,a_2,b_1,b_2 = n + 1}^{n + k}c_{a_1}c_{a_2}g^{a_2b_2}\partial_{b_2}(g^{a_1b_1}\partial_{b_1}f)\\
		&\vdots\\
		\partial_t^lf &= \sum_{a_1,\ldots,a_l,b_1,\ldots,b_l = n + 1}^{n + k}c_{a_1}\cdots c_{a_l}g^{a_lb_l}\partial_{b_l}(g^{a_{l - 1}b_{l - 1}}\partial_{b_{l - 1}}(\cdots g^{a_1b_1}\partial_{b_1}f)).
	\end{align*}
	The coefficients of the $c_a$ are smooth functions on $U$, and so in particular restrict to smooth functions on $\Sigma$; that is, in particular, functions of $x^i$ only.
	Now, by Taylor's Theorem, we may write
	\begin{equation*}
		\beta^*f(x,t,\omega) = f_0 + tf_1 + \cdots + t^lf_l + O(t^{l + 1}),
	\end{equation*}
	where $f_j = \frac{1}{j!}\partial_t^j(\beta^*f)|_{t = 0}$. The latter is the restriction to $t = 0$ of the unique smooth extension of $\frac{1}{j!}\beta^*\big( \partial_t^jf \big)$; so by the above, we see that it
	is a homogeneous polynomial in the $c_a$, with coefficients smooth functions on $\Sigma$. But each $c_a$ is itself in $\mathcal{H}_1(\widetilde{\Sigma})$, so the result now follows by Proposition
	\ref{hsigmaprop}.

	\ref{hypropsolve} Let $F \in \mathcal{Y}_i(\widetilde{\Sigma})$. Write $F = \sum_{j = 0}^iH_j$, where $H_j \in \mathcal{H}_j(\widetilde{\Sigma})$. But then observe that
	\begin{equation*}
		w = \sum_{j = 0}^{i}\frac{1}{\mu - j(j + k - 2)}H_j
	\end{equation*}
	satisfies $(\Delta_{\mathring{b}} + \mu)w = F$. Since $-\mu$ is not an eigenvalue of $\Delta_{\mathring{b}}$, such $w$ is unique.

	If $-\mu = \lambda_p$, then the same formula with the term corresponding to $j = p$ omitted still gives a solution in $\mathcal{Y}_i$. Any other solution will differ by some nonzero function in
	$\mathcal{H}_p(\widetilde{\Sigma})$, which will therefore not be in $\mathcal{Y}_i(\widetilde{\Sigma})$.
\end{proof}

In fact, the converse of Proposition \ref{hyprop}\ref{hypropdiff} is also true, and will be sufficiently important to us that we state it separately.

\begin{proposition}\label{smoothprop}
	Suppose $f \in C^{\infty}(\widetilde{M})$. Then $f = \beta^*u$ for some $u \in C^{\infty}(M)$ if and only if
	\begin{equation}\label{smoothcond}
		\left.\frac{\partial^jf}{\partial t^j}\right|_{t = 0} \in \mathcal{Y}_j(\widetilde{\Sigma})\text{ for all } j.
	\end{equation}
\end{proposition}
\begin{proof}
	For the remaining direction, suppose (\ref{smoothcond}) holds. Clearly, there exists $u \in C^{\infty}(M \setminus \Sigma)$ such that $f|_{\widetilde{M} \setminus \widetilde{\Sigma}} = \beta^*u$; so we need only
	show that $u$ extends smoothly to $M$. Let $j \geq 1$. By hypothesis and Taylor's theorem, we may write
	\[f = f_0 + tf_1 + \cdots + t^jf_j + t^{j + 1}H_{j + 1},\]
	where $H_{j + 1} \in C^{\infty}(\widetilde{M})$ and $f_i \in \mathcal{Y}_i(\widetilde{\Sigma})$. By Proposition \ref{hyprop}\ref{hypropsmooth}, we have $f_0 + tf_1 + \cdots + t^kf_k \in \beta^*C^{\infty}(M)$; let us denote
	this function by $\beta^*u_j$, where $u_j \in C^{\infty}(M)$. Thus, $u = u_j + t^{j + 1}h_j$ for some bounded function $h_j \in C^{\infty}(M \setminus \Sigma)$ with the property
	that $|\nabla^ph_j|_g = O(t^{-p})$. It follows that $u \in C^j(M)$. Since $j$ is arbitrary,
	this yields the result.
\end{proof}

We note that any first eigenfunction $u \in \mathcal{H}_1(\widetilde{\Sigma})$ defines a linear function $\Lambda u$ on $N\Sigma$. This is because it is the restriction of a \emph{linear} polynomial on $N\Sigma$, i.e., a one-form
along $\Sigma$.  Now, recall the definition (\ref{ca}) of $c_a$; on each $\spheres^{k - 1}$, it is just the restriction of $y^a$ to the sphere. In particular, then, the set $\left\{ c_a \right\}_{a = n + 1}^{n + k}$ spans
$\mathcal{H}_1(\widetilde{\Sigma})$. Thus, any element $u \in \mathcal{H}_1(\widetilde{\Sigma})$ may be written $u = u^ac_a$, the Einstein convention being in effect, where each $u^a$ is a function
on $\Sigma$ (in particular, constant on $\spheres^{k - 1}$). We then get the following useful formulas for the one-form pairing.
\begin{lemma}
	\label{oneformlem}
	Suppose that $u \in \mathcal{H}_1(\widetilde{\Sigma})$ and that $X \in N_p\Sigma$. Then $\Lambda u$ acts on $X$ by
	\begin{align}
		\Lambda u(X) &= kX^au(p)^b\fint_{SN_p\Sigma}c_ac_bdV_{\bo}\label{oneformlemeq1}\\
		&= k\fint_{SN_p\Sigma}u(Y)\langle Y,X\rangle dV_{\bo}(Y).\label{oneformlemeq2}
	\end{align}
\end{lemma}
\begin{proof}
	From the equations $\langle dy^a,dy^b\rangle = g^{ab}$ and $y^a = tc_a$, we may conclude that along $\widetilde{\Sigma}$,
	\begin{equation}
		\label{dcaeq}
		\langle dc_a,dc_b\rangle_{\mathring{b}} = g^{ab} - c_ac_b = \delta_{ab} - c_ac_b.
	\end{equation}
	Now, because $c_a \in \mathcal{H}_1(\widetilde{\Sigma})$, it thus follows from (\ref{dcaeq}) and the product rule for $\Delta_{\bo}(c_ac_b)$ that $c_ac_b - \frac{1}{k}\delta_{ab} \in \mathcal{H}_2(\widetilde{\Sigma})$.
	Thus,
	\begin{align*}
		X^au(p)^b\int_{SN_p\Sigma}c_ac_bdV_{\bo} &= \frac{1}{k}x^au(p)^b\int_{SN_p\Sigma}\delta_{ab}dV_{\bo}\\
		&= \sum_{a = n + 1}^{n + k}\frac{1}{k}X^au(p)^a\vol_{\bo}(\spheres^{k - 1})\\
		&= \sum_{a = n + 1}^{n + k}\frac{1}{k}\vol_{\bo}(\spheres^{k - 1})u^ac_a(X)\\
		&= \frac{1}{k}\vol_{\bo}(\spheres^{k - 1})\Lambda u(X).
	\end{align*}
	from which the first equation follows. The second formula follows quickly from the first.
\end{proof}

The results and notation of this section so far are for the case $k > 1$. For notational uniformity, we also define these notations for the case $k = 1$. In this case, the unit normal sphere is just
$\spheres^0$, which does not have differential operators defined on it. In this case, we simply define $\mathcal{H}_j(\widetilde{\Sigma})$ to be functions on $\widetilde{\Sigma}$ that are even
under the antipodal map if $j$ is even, and those that are odd if $j$ is odd. The $\mathcal{Y}_j(\widetilde{\Sigma})$ are now defined by $\mathcal{Y}_j(\widetilde{\Sigma}) = \mathcal{H}_j(\widetilde{\Sigma})$. 
Then the $\pi_j$ are still projections onto one of the two spaces. It remains the case that $\mathcal{H}_0(\widetilde{\Sigma})$ may be canonically identified with $C^{\infty}(\Sigma)$. In this setting,
we take (\ref{oneformlemeq1}) as the definition of $\Lambda u$, where the integral is now just a sum. Finally, in this case, we set $\Delta_{\bo} = 0$.

%% file: tex/smooth.tex
\section{Smoothness Theorems}\label{smoothsec}
In this section, we prove smoothness theorems about three important functions.

\subsection{Initial smoothness of the singular Yamabe function}\label{smoothsysec}
Recall that a singular Yamabe function on $M$ is a function $u$ such that the metric $g^+ = u^{-2}g$ has constant scalar curvature $(n + k - 1)(k - n - 2)$ and such that $u|_{\Sigma} = 0$.
By the usual conformal change formula for scalar curvature, this is equivalent to $u$ satisfying the PDE
\begin{equation}\label{leq}
	\begin{split}
		0 = 2L[u] &:= (n + 2 - k) - (n + k)|du|_g^2\\
		&\qquad\qquad\qquad+ 2u\Delta_gu + \frac{1}{n + k - 1}R_gu^2,\\
		u|_{\Sigma} &= 0.
	\end{split}
\end{equation}
If $n - k + 2 > 0$, so that the prescribed curvature is negative, it is a theorem of \cite{am88} that $u$ always exists and is unique.
In this case, it follows by the main result in \cite{maz91sy} that $\beta^*u$ is polyhomogeneous on $\widetilde{M}$; that is, $\beta^*u$ has an asymptotic expansion of the form
\begin{equation*}
	\beta^*u \sim \sum_{j = 1}^{\infty}\sum_{l = 0}^{p_j}\sum_{m = 0}^{q_{j,l}}a_{j,l,m}t^{j + \delta_l}(\log t)^m,
\end{equation*}
where we have $\delta_0 = 0$, $0 < \delta_l < 1$ (for $l > 0$), and $a_{j,l,m} \in C^{\infty}(\widetilde{\Sigma})$. It also follows from \cite{maz91sy} that $p_j = 0$ for $j \leq n$ and $q_{j,l} = 0$ for $j \leq n + 1$; 
that is, $\beta^*u$ is smooth on $\widetilde{M}$ up to the appearance of
a term of the form $a_{n + 2,1,0}t^{n + 1 + \delta}$ (where $\delta < 1$), with a log term following at exactly order $t^{n + 2}\log t$. Also, $a_{1,0,0} \equiv 1$. In the non-negatively curved case $n - k + 2 \leq 0$, we choose
to consider formal solutions, as discussed in the introduction.

It is our goal in this section to show that, in fact, up to the order $t^{n + \delta}$, the function $\frac{u}{t}$ is smooth on $M$ itself.
Since the two arguments are intimately related, we develop at the same time the existence and smoothness (on $M$) of formal expansions as in (\ref{formaleq}), thus also proving Theorem \ref{expandthm}. 
We turn for now to the development of such formal expansions, initially with no restriction on $n$ and $k$. 
Any exact solution that is polyhomogeneous will necessarily agree with this formal expansion, at least up to the first formally undetermined order.

We may write
\begin{equation*}
	u = tv = t(v_0 + tv_1 + \cdots).
\end{equation*}
We assume by hypothesis that $v_0 = 1$, and we restrict ourselves to formal solutions with this property. To construct a formal solution, we determine $v$ iteratively, order-by-order.   
To this end, we define for $s > 0$ the \emph{indicial operator} of order $s$, $I_s:C^{\infty}(\widetilde{\Sigma}) \to
C^{\infty}(\widetilde{\Sigma})$, by
\[I_s[\varphi] := t^{-s}(L[t(v + t^s\varphi)] - L[tv])|_{t = 0},\]
which is independent of $v$. A straightforward calculation shows that
\[I_s[\varphi] = \Delta_{\bo}\varphi + (s^2 - ns - (n - k + 2))\varphi.\]
It will also be useful to record the formula
\begin{equation}\label{logeq}
	\begin{split}
		L[t(v + t^{s}(\log t)^p\varphi)] - L[tv] =&\ t^s(\log t)^pI_s[\varphi]\\
		&+ p(2s - n)t^s(\log t)^{p - 1}\varphi\\
		&+ p(p - 1)t^s(\log t)^{p - 2}\varphi + o(t^s).
\end{split}
\end{equation}
In particular,
\begin{equation}\label{log1eq}
	L[t(v + t^s(\log t)\varphi)] - L[tv] = t^s(\log t)I_s[\varphi] + (2s - n)t^s\varphi + o(t^s).
\end{equation}

\begin{lemma}
	\begin{enumerate}[(a)]
		\item\label{partneg} If $n - k + 2 > 0$, then $I_s$ is injective for $0 \leq s \leq n$. 
		\item\label{partpos} If $n - k + 2 \leq 0$, then there are two values of $\gamma < n$ such that $I_{\gamma}$ fails to be injective. Specifically:
			\begin{enumerate}[(i)]
				\item Suppose $n$ is even. If $k \in E_n$ with $k = n + 2 + 2np - 4p^2$, then $I_{2p}$ and $I_{n - 2p}$ each have one-dimensional kernels. If $k \in O_n$ with
					$k = 2n + 1 + 2(n - 2)q - 4q^2$, then $I_{2q + 1}$ and $I_{n - 2q-1}$ both have one-dimensional kernels. The indicial operator is injective for all other values of $s$
					with $s \leq n$.

					If $k \notin E_n \cup O_n$, then $I_{\gamma}$ and $I_{n - \gamma}$ fail to be injective for some non-integral $\gamma \in (0,n)$.
				\item Suppose $n$ is odd.
					If $k \in E_n$ with $k = n + 2 + 2np - 4p^2$, then $I_{2p}$ and $I_{n - 2p}$ each have one-dimensional kernels. The indicial operator is injective for all other values of $s$
					with $s \leq n$. If $k = n + 2$, then $I_0$ and $I_n$ each have one-dimensional kernels.
					
					If $k \notin O_n$, then $I_{\gamma}$ fails to be injective for some non-integral $\gamma \in (0,n)$.
			\end{enumerate}
	\end{enumerate}
\end{lemma}
\begin{proof}
	If $n - k + 2 > 0$, then $s^2 - ns - (n - k + 2) < 0$ for $0 \leq s \leq n$, so the claim follows by Proposition \ref{hyprop}\ref{hypropsolve}.

	In order not to be injective, $I_s$ must be $\Delta_{\bo} - \lambda$ for some $\lambda \in \spec(\Delta_{\bo})$. Recall that the spectrum is $\left\{ \lambda_j \right\}$ with
	$\lambda_j = -j(j + k - 2)$. First, observe that for all $n, k$, we have $I_{n + 1} = \Delta_{\bo} - \lambda_1$. Since the other value of $s$ for which $I_s = \Delta_{\bo} - \lambda_1$ is
	$s = -1$, it follows that we can never have $s^2 - ns - (n - k + 2) = \lambda_j$ for $j \geq 2$ and $s \in [0,n]$. Thus, we need only consider when $I_s = \Delta_{\bo} - \lambda_0 = \Delta_{\bo}$.

	This occurs -- it is easy to see -- when 
	\begin{equation}
		s = \frac{n \pm \sqrt{n^2 + 4n -4k + 8}}{2}.
	\end{equation}
	Now for $k >> n$, this will never happen at all. For most pairs $n, k$, it will happen only at non-integer values.
	But we wish to characterize the values of $n,k$ for which integral values of $s$ are obtained, paying attention also to the parity of $s$ when this occurs.

	Integer roots, of course, will always occur in pairs. Suppose first that $n$ is even. Then $j$ and $n - j$ have the same parity, so both solutions will have the same parity. Let us find, say, the smaller:
	to find situations when $\frac{1}{2}(n - \sqrt{n^2 + 4n - 4k + 8})$ is an even integer, we set it equal to $2p \leq \frac{n}{2}$. Solving the equation, we find that $k \in E_n$, and conversely,
	every such $k$ gives a solution. We repeat similarly with $s = 2q + 1$ and obtain $O_n$.

	For odd $n$, every integer solution pair will come in an even/odd pair. Solving for the even one and repeating the arithmetic yields the claim.
\end{proof}

We will make use of the following nice fact about $L$.
\begin{lemma}\label{smoothlem}
	If $v \in C^{\infty}(M)$, then $L[tv] \in C^{\infty}(M)$.
\end{lemma}
\begin{proof}
	We have
	\begin{align*}
		2L[tv] &= (n + 2 - k) - (n + k)t^2|dv|_g^2 - (n + k)v^2\\
		&\quad - 2(n + k)tv\langle dt,dv\rangle + 2t^2v\Delta_gv + 2tv^2\Delta_gt\\
		&\quad + 4tv\langle dt,dv\rangle + \frac{1}{n + k - 1}R_gt^2v^2.
	\end{align*}
	The first three terms, the fifth term, and the last term are manifestly smooth since $t^2$ is smooth. Since $t \frac{\partial}{\partial t} = t\grad_gt = \frac{1}{2}\grad_g(t^2)$ is a smooth vector field, 
	the fourth and seventh terms
	are likewise smooth. This leaves only the sixth term to analyze. Its smoothness is equivalent to the smoothness of
	\begin{align}
		t\Delta_gt &= t\diver_g\grad_gt\notag\\
		&= \diver_g(t\grad_gt) - (\grad_gt)(t)\notag\\
		&= \diver_g(t\grad_gt) - 1.\label{tdelt}
	\end{align}
	Since, again, $t\grad_gt$ is a smooth vector field, this yields the claim.
\end{proof}

The following theorem immediately implies Theorem \ref{expandthm}, but its most important consequence
is that formal solutions $v$ are actually smooth functions on $M$ to $o(t^n)$
\begin{theorem}\label{usmooththm}
	Suppose $n \geq 2$.
	\begin{enumerate}[(a)]
		\item\label{noevenorodd} Suppose $k \notin (E_n \cup O_n)$. There exists a formal solution $v$ of the form
	\begin{equation}
		\label{vexp}
		v = 1 + tv_1 + t^2v_2 + \cdots + t^nv_n + t^{n + 1}(\log t)\mathcal{L} + t^{n + 1}v_{n + 1} + o(t^{n + 1})
	\end{equation}
	to the equation $L[tv] = O(t^{n + 2})$, where, $v_j,\mathcal{L} \in C^{\infty}(\widetilde{\Sigma})$. Such $v$ is unique to order $O(t^{n + 1})$.

	If $n$ is odd and $k > 1$, then $\mathcal{L} = 0$. In this case, $v$ may be made unique to order $o(t^{n + 1})$ by the condition that $\pi_1(v_{n + 1}) = 0$, which is a conformally invariant condition.

	If $n$ is even or $k = 1$, then $\mathcal{L}$ is a pointwise conformal invariant of weight $-(n + 1)$: if $\hat{g} = e^{2\omega}g$ is a conformally related metric and the expansion (\ref{vexp}) is made
	with respect to its distance function $\hat{t}$, then $\widehat{\mathcal{L}} = e^{-(n + 1)\omega}\mathcal{L}$.

	In all cases, if
	\begin{equation*}
		\bar{v} = 1 + tv_1 + t^2v_2 + \ldots + t^nv_n,
	\end{equation*}
	then $\bar{v} \in \beta^*C^{\infty}(M)$.

\item\label{oddcase} If $k \in O_n$, then the statement from \ref{noevenorodd} holds in all particulars, except that, in order to obtain uniqueness, we must require that 
	$\pi_0(v_{\nu}) = 0$ for each integer $\nu \in (0,n]$ for which $I_{\nu}$ fails to be injective. This is a conformally invariant condition.
	\item For $k \in E_n$, let $\nu \in (0,n]$ be the smallest nonzero integer for which $I_{\nu}$ fails to be injective. If $\nu \neq \frac{n}{2}$, then there exists a formal solution $v$ of the form
	\begin{equation}\label{logexpeq}
		v = 1 + tv_1 + \cdots + t^{\nu - 1}v_{\nu - 1} + t^{\nu}(\log t)A + t^{\nu}v_{\nu}
	\end{equation}
	to the equation $L[tv] = O(t^{\nu + 1})$, where $v_j, A \in C^{\infty}(\widetilde{\Sigma})$.
	If, on the other hand, $\nu = \frac{n}{2}$, there exists a formal expansion $v$ having the same form (\ref{logexpeq}) but with $(\log t)^2$ appearing in place of $\log t$.

	Moreover, $A$ is conformally invariant of weight $-\nu$.
	\end{enumerate}
\end{theorem}
We remark that the expansion (\ref{logexpeq}) could be continued arbitrarily high, but this would not be of present interest to us and is a standard argument. Of interest, however, is that for odd $n$ and with the
uniqueness-fixing condition in case \ref{noevenorodd}, the expansion could be extended \emph{uniquely} to infinite order, so long as $I_{\nu}$ remained injective for all $\nu$ (which depends on $n$ and $k$). Compare this
to the evenness condition for odd $n$ in the formal existence theory of Poincar\'{e}-Einstein metrics \cite{fg12}. See below for further discussion.
\begin{proof}
	First assume that $k \notin (E_n \cup O_n)$. We set $V_0 = v_0 = 1$. It is immediate that $L[tV_0] = O(t)$. Also, by Lemma \ref{smoothlem}, $L[tv_0] = tF_1 + O(t^2)$ is the pullback of a smooth function on $M$.
	We wish to solve 
	\begin{equation*}
		O(t^2) = L[tv_0 + t^2v_1] = t(I_1[v_1] + F_1) + O(t^2);
	\end{equation*}
	so we must solve $I_1(v_1) = -F_1$. Now, $I_1 = \Delta_{\bo} - (1 + 2n - k)$. Moreover, $F_1 \in \mathcal{Y}_1(\widetilde{\Sigma})$ by Proposition \ref{hyprop}\ref{hypropdiff}. 
	Now, $1 + 2n - k$ is not a spherical eigenvalue since $k \notin O_n$. Since $F_1 \in \mathcal{Y}_1(\St)$, it follows by
	by Proposition \ref{hyprop}\ref{hypropsolve} that there exists a unique $v_1 \in \mathcal{Y}_1(\widetilde{\Sigma})$ satisfying $I_1[v_1] = -F_1$. This $v_1$ is unique among all smooth functions.
	Set $V_1 = V_0 + tv_1$. By Proposition \ref{smoothprop}, $V_1$ is the pullback of a smooth function.

	For purposes of induction, assume that we have constructed $V_{j - 1}  = \sum_{i = 0}^{j - 1}t^iv_i \in \beta^{*}C^{\infty}(M)$ satisfying
	$L[tV_{j - 1}] = O(t^j)$ and unique. We may write $L[tV_{j - 1}] = t^{j}F_j + O(t^{j + 1})$. Since the right-hand side is smooth,
	we again have $F_j \in \mathcal{Y}_j(\widetilde{\Sigma})$. For each $j$, $I_{j}$ is an isomorphism on $C^{\infty}(\St)$ by hypothesis.
	By Proposition \ref{hyprop}\ref{hypropsolve}, we may uniquely solve $I_j[v_j] + F_j = 0$ for $v_j \in \mathcal{Y}_j(\widetilde{\Sigma})$. 
	Again by Proposition \ref{smoothprop}, we have $V_j \in \beta^*C^{\infty}(M)$. By induction, the procedure can be continued up through the construction of
	$V_n$; and $V_n \in \beta^*C^{\infty}(M)$, as claimed.

	Now, $L[tV_n] = t^{n + 1}F_{n + 1} + O(t^{n + 2})$; and $F_{n + 1} \in \mathcal{Y}_{n + 1}(\widetilde{\Sigma})$ by smoothness of $V_n$ and by Proposition \ref{smoothprop}. 
	To proceed further, we wish to solve $I_{n + 1}[v_{n + 1}] = -F_{n + 1}$. However, $I_{n + 1}$ is not injective, as it is $\Delta_{\bo} + k - 1$, and $1 - k$ is the first negative eigenvalue of $\Delta_{\bo}$. Now, if $n$
	is odd, then $n + 1$ is even, and in particular is the opposite parity from $1$; so by Proposition \ref{hyprop}\ref{hypropsolve}, if $k > 1$ there still exists a unique solution
	$v_{n + 1}$ in $\mathcal{Y}_{n + 1}(\widetilde{\Sigma})$; other solutions will exist, but (unlike this one) will not satisfy $\pi_1(v_{n + 1}) = 0$. Moreover, by Proposition \ref{smoothprop}, this condition is the unique choice
	for which we will have $V_{n + 1} \in \beta^*C^{\infty}(M)$. Conformal invariance of the property $\pi_1(v_{n + 1}) = 0$ follows from this fact and uniqueness, 
	since if $\hat{g} = e^{2\omega}g$, then $e^{\omega}u$ will be a formal solution for $\hat{g}$; and as we will show in section \ref{distfuncsec}, $\frac{\hat{t}}{t}$ is smooth on $M$.

	If $n$ is even, on the other hand, then $I_{n + 1}[v_{n + 1}] = -F_{n + 1}$ generally cannot be solved, as $\pi_1(F_{n + 1}) = 0$ will generally not hold. In this case, we must add a log term to subtract off
	$\pi_1(F_{n + 1})$. Let $\mathcal{L} = -\frac{1}{n + 2}\pi_1(F_{n + 1})$, and set $V_{n + 1,-1} = V_{n} + t^{n + 1}(\log t)\mathcal{L}$. It follows from (\ref{log1eq}) that
	$L[tV_{n + 1,-1}] = L[tV_{n}] - \pi_1(F_{n + 1})$. Thus, writing $L[tV_{n + 1,-1}] = t^{n + 1}F_{n + 1,-1}$, we have $\pi_1(F_{n + 1,-1}) = 0$. We may now proceed as before,
	adding a term $t^{n + 1}v_{n + 1}$ to obtain $V_{n + 1} = V_{n + 1,-1} + t^{n + 1}v_{n + 1}$ satisfying $L[tV_{n + 1}] = o(t^{n + 2})$. Note that we could fix $v_{n + 1}$ uniquely by requiring
	the condition that $\pi_1(v_{n + 1}) = 0$, but this condition is not conformally invariant: it is not necessary for smoothness, since $n + 1$ and $1$ have the same parity; and one may easily compute by expansion that,
	in general, it is not preserved under conformal change.

	If $n$ is odd and $k = 1$, we can solve through $V_n$ as before. Now, $L[tV_n] = t^{n + 1}F_{n + 1}$, and by smoothness, $F_{n + 1} \in \mathcal{Y}_{n + 1}(\widetilde{\Sigma})$, 
	which is to say, $F_{n + 1}$ is even.
	Thus, in this case, $\pi_1(F_{n + 1}) = 0$ does hold. However, for $k = 1$, the indicial operator $I_{n + 1}$ vanishes identically, so we still must add a log term as before. The above paragraph
	thus applies, but we note that $\mathcal{L}$ will be an even function and, while $v_{n + 1}$ is formally undetermined, we may impose the conformally invariant choice that it be even (which
	is equivalent to requiring that $v$ be smooth through this order). This choice does \emph{not} give uniqueness, however, as the even part of $v_{n + 2}$ is still entirely undetermined.

	It remains to prove the conformal invariance property of $\mathcal{L}$. Let $\hat{g} = e^{2\omega}g$, and suppose that we have
	\begin{equation*}
		\hat{v} = 1 + \hat{t}\hat{v}_1 + \ldots + \hat{t}^{n}\hat{v}_n + \hat{t}^{n + 1}(\log\hat{t})\widehat{\mathcal{L}} + \hat{t}^{n + 1}\hat{v}_{n + 1} + o(\hat{t}^{n + 1}).
	\end{equation*}
	By formal uniqueness and the conformal invariance of the equation, we must have $\hat{u} = e^{\omega}u + O(t^{n + 1})$. As shown in section \ref{distfuncsec},
	$\left.\frac{\hat{t}}{t}\right|_{t = 0}=\left.e^{\omega}\right|_{t = 0}$. It follows that we must have $t^{n + 1}(\log t)\mathcal{L} = \hat{t}^{n + 1}(\log \hat{t})\widehat{\mathcal{L}} +
	o(t^{n + 1}\log t)$, from which it follows that $\widehat{\mathcal{L}} = e^{-(n + 1)\omega}\mathcal{L}$, as desired.

	This concludes case \ref{noevenorodd}. The remaining cases are minor variations, which we will briefly describe. If $k \in O_n$, then there will be some odd integer $\nu < n + 1$ for which
	$I_{\nu} = \Delta_{\bo}$, whose kernel consists of functions constant on each fiber (i.e., restricting to the zeroth eigenfunctions). 
	However, because $\nu$ and $0$ have opposite parity, by \ref{hyprop}\ref{hypropsolve}, we may still
	uniquely solve the indicial equation subject to the smoothness condition $\pi_0(F_{\nu}) = 0$. This yields \ref{oddcase}.

	Now suppose that $k \in E_n$. The argument proceeds exactly the same until we reach $j = \nu$, where $\nu$ is the smallest positive order at which $I_{\nu}$ is not injective. That is,
	$I_{\nu} = \Delta_{\bo}$. This time, because $\nu$ is even, we cannot necessarily solve $I_{\nu} = -F_{\nu}$, because $F_{\nu} \in \mathcal{Y}_{\nu}(\St)$ and so $\pi_0(F_{\nu})$ may be not vanish.
	Assuming for now that $\nu \neq \frac{n}{2}$, we may proceed exactly as we did above for $\mathcal{L}$, except that we will add $t^{\nu}(\log t)A$ with $A = -\frac{1}{2\nu - n}\pi_0(F_{\nu})$. The conformal invariance
	argument proceeds the same.

	Now suppose $\nu = \frac{n}{2}$.
	Because the second term on the right-hand side of (\ref{log1eq}) vanishes, we may no longer fix the possible obstruction by adding a log. We must instead let
	$A = -\frac{1}{2}\pi_0(F_{\frac{n}{2}})$, and $V_{\frac{n}{2},-2} = V_{\frac{n}{2} - 1} + t^{\frac{n}{2}}(\log t)^2A$. The argument now proceeds in precisely the same way, including conformal invariance of $A$.
\end{proof}
We note that if $n$ is odd and we do \emph{not} impose the uniqueness-fixing condition $\pi_1(v_{n + 1}) = 0$, it is still the case that the space of possible coefficients of $t^{n + 1}$ is just
$v_{n + 1} + \mathcal{H}_1(\widetilde{\Sigma})$. This is important to our variational analysis.

In case $n - k + 2 > 0$, it follows  easily from polyhomogeneity that $v_0 = 1$, although a deeper argument of \cite{ln74} (see \cite{maz91sy}) shows it (still in the negative-curvature case) without the assumption
of polyhomogeneity. We summarize the application of Theorem \ref{usmooththm} to this case, via \cite{maz91sy}.
\begin{corollary}\label{negcol}
	Suppose $n - k + 2 > 0$. Then the unique solution $u$ to $L[u] \equiv 0$ with $u = 0$ on $\St$ may be written $u = tv$ with 
	\begin{equation*}
		v = \bar{v} + t^{n + \delta}v_{n + \delta} + t^{n + 1}(\log t)\mathcal{L} + t^{n + 1}v_{n + 1} + v_{phg},
	\end{equation*}
	where
	$\delta \in (0,1)$, $\bar{v} \in \beta^*C^{\infty}(M)$, $v_{n + \delta} \in C^{\infty}(\widetilde{\Sigma})$,
	and $v_{phg} = o(t^{n + 1})$ is polyhomogeneous. Also, $\mathcal{L} \in \mathcal{H}_1(\widetilde{\Sigma})$ if $k \neq 1$, and $\mathcal{L} \in \mathcal{H}_{n + 1}(\widetilde{\Sigma})$ if $k = 1$.
	Moreover, $\mathcal{L}$ is conformally invariant of weight $-(n + 1)$ and vanishes if $n$ is odd and $k > 1$. Also, if $n$ is odd,
	we may write $v_{n + 1} = \tilde{v}_{n + 1} + \mathcal{A}$, where $\tilde{v}_{n + 1} \in \mathcal{H}_1(\widetilde{\Sigma})$ is globally determined and conformally invariant of weight $-(n + 1)$ and $\mathcal{A} \in
	\mathcal{Y}_{n + 1}(\widetilde{\Sigma})$. Finally, $v_{n + \delta}$ is globally determined and conformally invariant of weight $-(n + \delta)$.
\end{corollary}
\begin{proof}
	The only claim that is not immediate is the conformal invariance of $\tilde{v}_{n + 1}$ in case $n$ is odd and $k > 1$. The possible existence of such a term in $\mathcal{H}_1(\widetilde{\Sigma})$ at this order
	follows from Theorem \ref{usmooththm} and the remark immediately following it. Recall that $u$ is unique and globally determined. Since $n$ is odd, $\mathcal{L} = 0$. Thus, if we write
	$\hat{g} = e^{2\omega}g$, then $\hat{u} = e^{\omega}u$; that is, $\hat{v} = e^{\omega}\left( \frac{t}{\hat{t}} \right)v$. Now, $e^{\omega}$ is smooth on $M$, and by Theorem \ref{tsmooththm} below,
	so is $\frac{t}{\hat{t}}$. It of course quickly follows as above that 
	$\hat{t}^{n + \delta}\hat{v}_{n + \delta} = t^{n + \delta}v_{n + \delta} + o(t^{n + 1})$. Now, due to the presence of $\tilde{v}_{n + 1}$, the function
	$v - t^{n + \delta}v_{n + \delta}$ is not smooth to order $o(t^{n + 1})$ (because $v_{n + 1}$ contains a harmonic, $\tilde{v}_{n + 1}$, of the wrong parity). On the other hand,
	$v - t^{n + \delta}v_{n + \delta} - t^{n + 1}\tilde{v}_{n + 1}$ is smooth to order $o(t^{n + 1})$. Thus, considering the order-by-order expansions of
	$\hat{v} = e^{\omega}\left( \frac{t}{\hat{t}} \right)v$, and noting again that $\left.\frac{t}{\hat{t}}\right|_{t = 0} = \left.e^{-\omega}\right|_{t = 0}$, we see that the only possible source of the term
	$\hat{t}^{n + 1}\hat{\tilde{v}}_{n + 1}$ is $t^{n + 1}\tilde{v}_{n + 1} = e^{(n + 1)\omega}\hat{t}^{n + 1}\tilde{v}_{n + 1} + o(t^{n + 1})$.

	If $k = 1$, then as discussed in the proof of Theorem \ref{usmooththm}, the parity of $\mathcal{L}$ will be opposite that of $n$.
\end{proof}
We note that, although $\mathcal{L}$ is conformally invariant of weight $-(n + 1)$, the one-form $\Lambda\mathcal{L}$ determined by it is invariant of weight $-n$, as may be seen from formula (\ref{oneformlemeq2}), since
the metric under conformal change contributes a factor of $e^{2\omega}$, while $Y$ contributes $e^{-\omega}$. The same applies for the one-form $u_{n + 2}$ defined by (\ref{uform}).

In Theorem \ref{usmooththm}\ref{noevenorodd}, in the case of $n$ odd, a uniqueness condition is stated which allows the formal expansion to be carried uniquely to one order higher. In fact, however, much more can be said. With this one
condition, the expansion is actually unique to infinite order, \emph{if} none of the higher indicial roots occurs at integer order. It is also true that if they occur at integer order, but the degree of the root and the degree of
the order have opposite parity, then additional such conditions may be imposed that yield uniqueness to infinite order. Although we will not need the result in this paper, it is convenient for completeness to record
the following facts here, since the proofs are exactly the same as for Theorem \ref{usmooththm}, and extensions to this work currently in progress will use these results.

In stating these results, the decisive condition is whether one can solve $s^2 - ns - (n - k + 2) = -\lambda_l = l(l + k - 2)$ for some integer $s$. If $k \notin E_n$, then we need only consider the case $s > n$, so we write
$s = n + \mu$. We therefore let
\begin{equation*}
	\begin{split}
		S_n = \left\{ p \in \mathbb{N}: (\mu - 1)(\mu + n + 1) = (l - 1)\right.(l +& p - 1) \text{ has solutions } \\
		&\left.(\mu,l) \in \mathbb{N}^2 \setminus \{(1,1)\} \right\}.
	\end{split}
\end{equation*}
Note that $E_n$ and $S_n$ are disjoint.
\begin{proposition}\label{smoothexprop}
	Let $n$ be odd and $k \notin E_n$.
	\begin{enumerate}[(a)]
		\item\label{easysmoothcase} Suppose $k \notin S_n$. There is a solution to $L[t\bar{v}] = O(t^{\infty})$ with $\bar{v} \in C^{\infty}(M)$ and $\bar{v}|_{t = 0} = 1$. Subject to the uniqueness condition given in
			Theorem \ref{usmooththm}\ref{noevenorodd}, such $\bar{v}$ is unique to infinite order.
		\item \label{solncase}Suppose $k \in S_n$. Let $\left\{ (\mu_j,l_j) \right\}$ be a lexicographic enumeration of
	\begin{equation*}
		P = \left\{ (\mu,l) \in \mathbb{N}^2\setminus\{(1,1)\}:(\mu - 1)(\mu + n + 1) = (l - 1)(l + k - 1)\right\}.
	\end{equation*}
	It is easy to see that for any $\mu$, there is at most one $l$ such that $(\mu,l) \in P$.
	\begin{enumerate}[(i)]
		\item \label{smoothsolncase}If $\mu_j \equiv l_j$ (mod $2$) for every $j$, then there is a solution to $L[t\bar{v}] = O(t^{\infty})$ with $\bar{v} \in C^{\infty}(M)$ and with
			$\bar{v}|_{t = 0} = 1$. For each $(\mu_j,l_j) \in P$, there is a conformally invariant smoothness condition, and $\bar{v}$ is unique mod $O(t^{\infty})$ because it satisfies
			each of these smoothness conditions.
		\item If $j$ is the first instance of $\mu_j \not\equiv l_j$ (mod 2), then there is a solution to the equation $L[u] = O(t^{n + \mu_j + 1})$ of the form $u = t\bar{v} + O(t^{n + \mu_j + 1}\log(t))$, with
			$\bar{v} \in C^{\infty}(M)$ and $\bar{v}|_{t = 0} = 1$. Such $\bar{v}$ is unique mod $O(t^{n + \mu_j})$.
	\end{enumerate}
	\end{enumerate}
\end{proposition}
The proof of the proposition is simply by continuing that of Theorem \ref{usmooththm}. We note that when $k = n + 2$, we are in the case of Proposition \ref{smoothexprop}\ref{solncase}\ref{smoothsolncase}, 
so a smooth formal solution exists, unique to infinite order.

\subsection{Smoothness of the conformal change-of-distance function}\label{distfuncsec}

In general, if $\hat{g}$ and $g$ are two Riemannian metrics with distance functions to $\Sigma$ given by $\hat{t}$ and $t$, then $\frac{\hat{t}}{t}$ may not be a smooth function on $M$. For example, the metric
$dx^2 + dxdy + dy^2$ on $\mathbb{R}^2$ leads to such a situation with respect to distance to the origin. However, if the two metrics are conformally related, then this situation does not arise.
\begin{theorem}\label{tsmooththm}
	Suppose $(M^{n + k},[g])$ is a conformal manifold, and $\Sigma^n$ is a closed embedded submanifold. Let $g, \hat{g} = e^{2\omega}g \in [g]$, and let $t, \hat{t}$ be the corresponding distance functions to $\widetilde{\Sigma}$.
	Then $\frac{\hat{t}}{t}$ extends to a smooth function on a neighborhood of $\Sigma$ in $M$, with $\left.\frac{\hat{t}}{t}\right|_{\Sigma} = e^{\omega}|_{\Sigma}$.
\end{theorem}
\begin{proof}
	Abusing notation, we also denote by $t$ and $\hat{t}$ the lifts of these functions to $\widetilde{M}$. They are both smooth defining functions of $\widetilde{\Sigma}$ in $\widetilde{M}$, since
	$\hat{t}$ is the polar distance function defined by another smooth set of Fermi coordinates on $M$ (see \cite{mel08}).
	We write $\hat{t} = \Psi t$ for some function $\Psi$ defined near $\Sigma$ (without loss of generality, defined on our tubular neighborhood). Then $\beta^*\Psi$ is a smooth function on $\widetilde{M}$ (near
	$\widetilde{\Sigma}$). 
	We write
	\begin{equation*}
		\wtp := \beta^*\Psi = \Psi_0 + t\Psi_1 + \cdots
	\end{equation*}
	Now, both $t$ and $\hat{t}$ satisfy the eikonal equation with respect to their associated metrics: $|dt|^2_g = 1 = |d\hat{t}|^2_{\hat{g}}$. Thus, away from $\Sigma$, we have
	\begin{align*}
		1 = |d\hat{t}|^2_{\hat{g}} &= e^{-2\omega}|d\hat{t}|^2_g\\
		&= e^{-2\omega}(\Psi^2|dt|^2_g + 2t\Psi \langle dt,d\Psi\rangle_g + t^2|d\Psi|^2_g)\\
		&= e^{-2\omega}(\Psi^2 + 2t\Psi \partial_t\Psi + t^2|d\Psi|^2_g).
	\end{align*}
	We now express this equation on the blowup in terms of the metric (\ref{gpolar}), writing $\tilde{\omega} = \beta^*\omega$:
	\begin{align*}
		1 &= e^{-2\tilde{\omega}}(\wtp^2 + 2t\widetilde{\Psi}\partial_t\wtp + t^2(\partial_t\wtp)^2 + |d\wtp|^2_{\bo}) + O(t^2).
	\end{align*}
	Writing $\tilde{\omega} = \omega_0 + t\omega_1 + \cdots$ and taking $t = 0$ gives
	\begin{equation}\label{eikonrestricteq}
		e^{2\tilde{\omega}_0} = \Psi_0^2 + |d\Psi_0|^2_{\bo}.
	\end{equation}
	Because $t$ and $\hat{t}$ are both defining functions for $\widetilde{\Sigma}$, the extension of $\widetilde{\Psi} = \frac{\hat{t}}{t}$ to $\widetilde{\Sigma}$ must be everywhere positive there. (If it had a zero,
	$d\hat{t}$ would vanish at that point.) Consider the restriction of $\Psi_0$ to a single fiber. It has a maximum and a minimum there by compactness, at each of which, $|d\Psi_0|_{\bo}$ vanishes.
	But since $e^{2\tilde{\omega}_0}$ is constant on the fiber (by smoothness of $\omega$ on $M$) and $\Psi_0 > 0$, we conclude by (\ref{eikonrestricteq}) that $\Psi_0$ is the same at both of these. Thus,
	$\Psi_0$ is a constant on each fiber; in particular, $\Psi_0 \equiv e^{\omega_0}$.

	We now study the eikonal operator in more detail. We define
	\begin{equation*}
		E[f] = f^2 + 2tf\partial_tf + t^2|df|^2_g - e^{2\omega}.
	\end{equation*}
	Plainly, since $t\partial_t$ is a smooth vector field on $M$, this operator preserves smoothness. We write $\widetilde{E}$ for the same operator on $\widetilde{M}$. We define an indicial operator
	associated to $E$, to compute how perturbations of $f$ at order $s \geq 1$ affect $E[f]$ at the same order, assuming $f_0 = f|_{\widetilde{\Sigma}}$ is positive and constant on fibers. Let
	$I_{f,s}:C^{\infty}(\widetilde{\Sigma}) \to C^{\infty}(\widetilde{\Sigma})$ be given by
	\begin{equation*}
		I_{f,s}[\varphi] := \left.t^{-s}\left( E[f + t^s\varphi] - E[f]  \right)\right|_{t = 0}.
	\end{equation*}
	We compute
	\begin{align*}
		E[f + t^s\varphi] &= f^2 + 2t^sf\varphi + t^{2s}\varphi^2 + 2t(f + t^2\varphi)(\partial_t f + st^{s - 1}\varphi)\\
		&\quad + t^2(\partial_tf + st^{s - 1}\varphi)^2 + |d(f + t^2\varphi)|^2_{\bo}.
	\end{align*}
	Because $s \geq 1$ and $|df|_{b} = O(t)$, we conclude
	\begin{equation*}
		I_{f,s}[\varphi] = 2(1 + s)f_0\varphi.
	\end{equation*}
	Since $\frac{1}{f_0} \in \mathcal{Y}_0(\widetilde{\Sigma})$, it follows that $I_{f,s}$ is an automorphism of $\mathcal{Y}_j(\widetilde{\Sigma})$ for all $s \geq 1$ and all $j$.

	We now let $\psi_j = \Psi_0 + t\Psi_1 + \cdots + t^j\Psi_j$. It is clear that $\psi_0 = \Psi_0$ is the pullback of a smooth function on $M$, since $\Psi_0$ depends only on the $\Sigma$ coordinates. 
	Thus, by the earlier observation that $E$ preserves smooth
	functions, $\widetilde{E}[\psi_0]$ is also the pullback of a smooth function on $M$. Let $F_1 = t^{-1}\widetilde{E}[\psi_0]|_{t = 0} \in C^{\infty}(\widetilde{\Sigma})$. Then
	$F_1 \in \mathcal{Y}_1(\widetilde{\Sigma})$ by smoothness and Proposition \ref{hyprop}\ref{hypropsmooth}. Since $\widetilde{E}[\hat{t}] = 0$, we must have
	\begin{equation*}
		O(t^2) = E[\Psi_0 + t\Psi_1] = t(F_1 + I_{\widetilde{\Psi}_0,1}[\Psi_1]) + O(t^2).
	\end{equation*}
	Thus, $I_{\wtp_0,1}[\Psi_1] = -F_1$, or $\Psi_1 = -e^{-\omega_0}F_1 \in \mathcal{Y_1}(\widetilde{\Sigma})$. It now follows from Proposition \ref{hyprop}\ref{hypropdiff} that $\psi_1 \in \beta^*C^{\infty}(M)$.
	We now proceed by induction, repeating this argument at each order, just as in the proof of Theorem \ref{usmooththm}.
\end{proof}

\subsection{Smoothness of determinant ratio}\label{detsmoothsec}
Finally, and more briefly, we show that on $\mathcal{U}$, the function $\sqrt{\frac{\det h \det \alpha}{\det h_0 \det \alpha_0}} = \sqrt{\frac{\det h \det \alpha}{\det h_0 \det \bo}}$ is a smooth function.
Observe that for any two metrics $g, G$, the ratio $\sqrt{\frac{\det g}{\det G}}$ is well-defined independent of the coordinate system in which $g, G$ are both expressed. 
In particular, if it defines a smooth function in one system of coordinates, then it defines one in another, even if the latter is not a smooth coordinate system. Now, we express the metric $g$ in (smooth) Fermi coordinates
as
\begin{align*}
	g &= h_{ij}(x^1,\cdots,x^n,y^{n + 1},\cdots,y^{n + k})dx^idx^j + 2\ell_{ai}dy^adx^i \\
	&\quad+ k_{ab}(x^1,\cdots,x^n,y^{n + 1},\cdots,y^{n + k})dy^ady^b,
\end{align*}
where this may be taken to define $k$ and $\ell$. We now define another smooth metric $G$ on $\mathcal{U}$ by
\begin{equation*}
	G = h_{ij}(x^1,\cdots,x^n,0,\cdots,0)dx^idx^j + k_{ab}(x^1,\cdots,x^n,0,\cdots,0)dy^ady^b.
\end{equation*}
Thus, $G$ is just the metric along $\Sigma$ ``translated'' along the flows of the normal Fermi coordinates. Since both are smooth, the function $\sqrt{\frac{\det g}{\det G}}$ is smooth. But, expressed in cylindrical coordinates, this is
just
\begin{equation*}
	\sqrt{\frac{\det g}{\det G}} = \sqrt{\frac{t^{2(k - 1)}\det h \det \alpha}{t^{2(k - 1)}\det h_0 \det \alpha_0 }} = \sqrt{\frac{\det h \det \alpha}{\det h_0 \det \bo}}.
\end{equation*}
The right-hand side is thus smooth.

%% file: tex/vol.tex
\section{Renormalized volume}\label{volsec}

In this section, we analyze the volume expansion for a (formal) singular Yamabe metric. Throughout, we assume that either $n - k + 2 > 0$ and $u$ is the unique singular Yamabe function, which satisfies
(\ref{polyhom}); or that $n - k + 2 \leq 0$ with $k \notin E_n$, and $u$ is a formal singular Yamabe function, satisfying (\ref{formaleq}). We handle both cases simultaneously.

\begin{proof}[Proof of Theorem \ref{mainthm}] We wish to expand the expression $\int_{\{t > \varepsilon\}}dV_{g^+}$ in powers of $\varepsilon$. The outline of our approach will follow \cite{g99}. We begin by studying the volume form $dV_{g^+}$ near
$\Sigma$ -- that is, on $\mathcal{T}$. We may there write
\begin{equation*}
	dV_{\gp} = \vartheta dt dV_{h_0}dV_{\bo},
\end{equation*}
where $\vartheta = u^{-n - k}\sqrt{\frac{\det g}{\det h_0 \det \bo}}$. 

Since $u = t\bar{v} + o(t^{n + 1})$ with $\bar{v}$ smooth, by Corollary \ref{negcol} or Theorem \ref{usmooththm}; 
and $\sqrt{\frac{\det g}{\det h_0\det\bo}} = t^{k - 1}\sqrt{\frac{\det h\det b}{\det h_0\det \bo}} \in t^{k - 1}C^{\infty}(M)$, by Section \ref{detsmoothsec};
we conclude that $t^{n + 1}\vartheta$ is smooth to order $o(t^{n})$.
Thus, by Proposition \ref{hyprop}\ref{hypropdiff},
\begin{equation*}
	t^{n + 1}\vartheta = \vartheta_0 + \cdots + t^n\vartheta_n + o(t^{n}),
\end{equation*}
where $\vartheta_j \in \mathcal{Y}_j(\widetilde{\Sigma})$. Hence,
\begin{equation*}
	\vartheta = t^{-n - 1}\vartheta_0 + t^{-n}\vartheta_1 + \cdots + t^{-1}\vartheta_n + o(t^{-1}).
\end{equation*}
Once again following \cite{g99}, we choose some small $t_0 > 0$ and write $\vol_{g^+}(\left\{ t > \varepsilon \right\}) = C + \int_{\varepsilon < t < t_0}dV_{g^+}$. We obtain
\begin{equation*}
	\vol_{g^+}\left(\left\{ t > \varepsilon \right\}\right) = c_0\varepsilon^{-n} + c_1\varepsilon^{1 - n} + \cdots + c_{n - 1}\varepsilon^{-1} + \mathcal{E}_{n,k}\log\left( \frac{1}{\varepsilon} \right)
	+ V_{n,k} + o(1).
\end{equation*}
Here,
\begin{align*}
	c_j &= \frac{1}{j - n}\int_{\widetilde{\Sigma}}\vartheta_j dV_{\bo}dV_{h_0}, \quad 0 \leq j \leq n - 1,\\
	\intertext{and}
	\mathcal{E}_{n,k} &= \int_{\widetilde{\Sigma}}\vartheta_ndV_{\bo}dV_{h_0}.
\end{align*}
Recall that $\pi_i:\mathcal{Y}_j(\widetilde{\Sigma}) \to \mathcal{H}_i(\widetilde{\Sigma})$ is projection on the $i$th summand. Since
\begin{equation*}
	\vartheta_j = \pi_0(\vartheta_j) + \pi_1(\vartheta_j) + \cdots + \pi_n(\vartheta_j)
\end{equation*}
and harmonics other than $0$th-degree integrate to zero,
\begin{align}
	c_j &= \frac{\vol_{\bo}(\spheres^{k - 1})}{j - n}\int_{\Sigma}\pi_0(\vartheta_j)dV_{h_0}\nonumber\\
	\mathcal{E}_{n,k} &= \vol_{\bo}(\spheres^{k - 1})\int_{\Sigma}\pi_0(\vartheta_n)dV_{h_0}.\label{eeq}
\end{align}
But for $j$ odd, $\pi_0(\vartheta_j) = 0$, by definition of $\mathcal{Y}_j$. Thus, $c_j = 0$ for odd $j$ and $\mathcal{E}_{n,k} = 0$ for odd $n$.

We now turn to conformal invariance.
	Let $\hat{g} = e^{2\omega}g$. Let $\hat{t}$ be the $\hat{g}$-distance function to $\Sigma$. Throughout the remainer of the proof, we use the identification (\ref{fermpolident}) induced by $g$ (and thus $t$). Thus,
	$\hat{t}$ is a function of $(x,\omega,t)$, for example. In particular, by Theorem \ref{tsmooththm}, we may write $t = \hat{t}\Psi$, where $\Psi$ is a smooth function on $M$; we may write, in particular,
	$t = \hat{t}\Psi(x,\omega,\hat{t})$.
	For fixed $(x,\omega) \in \widetilde{\Sigma}$, set $\hat{\varepsilon} = \varepsilon\Psi(x,\omega,\varepsilon)$. Then $\hat{t} > \varepsilon$ if and only if $t > \hat{\varepsilon}$. We are now in a position
	to compare the volume expansions:
	\begin{align}
		\vol_{g^+}&\left( \left\{ t> \varepsilon \right\} \right) - \vol_{g^+}\left( \hat{t} > \varepsilon \right) = \vol_{g^+}\left( \left\{ t > \varepsilon \right\} \right) - 
		\vol_{g^+}\left( \left\{ t > \hat{\varepsilon} \right\} \right)\notag\\
		&= \int_{\widetilde{\Sigma}}\int_{\varepsilon}^{\hat{\varepsilon}}\vartheta dtdV_{h_0}dV_{\bo}\notag\\
		&= \int_{\widetilde{\Sigma}}\left( \frac{\vartheta_0}{n}(\varepsilon^{-n} - \hat{\varepsilon}^{-n}) + \cdots + \frac{\vartheta_{n - 1}}{1}(\varepsilon^{-1} - \hat{\varepsilon}^{-1})\right.\notag\\
		&\qquad\qquad+\left.\vartheta_n\log\left( \frac{\hat{\varepsilon}}{\varepsilon} \right)\right)dV_{h_0}dV_{\bo} + o(1)\notag\\
		&= \sum_{j = 0}^{n - 1}\frac{\varepsilon^{j - n}}{j - n}\int_{\widetilde{\Sigma}}\vartheta_j(1 - \Psi^{j-n})dV_{\bo}dV_{h_0} + \int_{\widetilde{\Sigma}}\vartheta_n\log\Psi dV_{\bo}dV_{h_0} + o(1).\label{volexpandeq}
	\end{align}
	Now, $\Psi$ is smooth, and $\Psi|_{t = 0} = e^{-\omega}|_{t = 0}$. Let $\omega_0 = \omega|_{t = 0}$. Thus, by Proposition \ref{hyprop}\ref{hypropdiff},
	\begin{align*}
		\Psi^{-j} &= e^{j\omega_0} + tF_{j,1} + t^2F_{j,2} + \cdots + t^{n}F_{j,n} + O(t^{n + 1}) \text{ and}\\
		\log\Psi &= -\omega_0 + t\ell_1 + \cdots + t^n\ell_n + O(t^{n + 1}),
	\end{align*}
	for some $F_{j,m},\ell_m \in \mathcal{Y}_m(\widetilde{\Sigma})$.

	It is clear that the right-hand side of (\ref{volexpandeq}) has no $\log\left( \frac{1}{\varepsilon} \right)$ term. Thus, $\widehat{\mathcal{E}}_{n,k} = \mathcal{E}_{n,k}$, which establishes the claim in case $n$ is even.
	Also, the $\varepsilon$-independent term in (\ref{volexpandeq}) is
	\begin{equation}\label{vdiffeq}
		V_{n,k} - \widehat{V}_{n,k} = \sum_{j = 0}^{n - 1}\frac{1}{j - n}\int_{\widetilde{\Sigma}}\vartheta_jF_{n - j,n - j}dV_{\bo}dV_{h_0} - \int_{\widetilde{\Sigma}}\vartheta_n\omega_0dV_{\bo}dV_{h_0}.
	\end{equation}
	Now, by Proposition \ref{hyprop}\ref{hypropprod}, $\vartheta_jF_{n - j,n-j} \in \mathcal{Y}_n(\widetilde{\Sigma})$. Similarly, we have $\omega_0 \in \mathcal{Y}_0(\widetilde{\Sigma})$, so $\omega_0\vartheta_n \in \mathcal{Y}_n
	(\widetilde{\Sigma})$. When $n$ is odd, the integral of each term therefore vanishes fiberwise. Consequently, in this case, $V_{n,k} = \widehat{V}_{n,k}$.

	Finally, the claim that $\mathcal{E}_{n,k}$ is independent of which singular Yamabe metric $g^+$ is chosen is due to (\ref{eeq}), since $\vartheta_n$ is determined by the formal asymptotics of $u$.
\end{proof}

We remark that the above proof would work if $\vartheta$ had arbitrary irrational powers of $t$ in its expansion. The terms these would contribute in (\ref{volexpandeq}) would not contribute
to (\ref{vdiffeq}). The theorem thus also applies to polyhomogeneous solutions of the singular Yamabe problem in any non-exceptional codimension.

In case $n$ is even, (\ref{vdiffeq}) gives the so-called conformal anomaly in $V_{n,k}$. The linear part, in particular, is simply $\vartheta_n$. Compare \cite{g99}.

%% file: tex/varsmooth.tex
\section{Smoothness under variation}\label{varsmoothsec}

Throughout this section and the next, we now consider a variation
$\mathcal{F}:(-\delta,\delta)_s \times \Sigma \to M$ of $\Sigma$ in $M$
satisfying $\mathcal{F}(0,\cdot) = \id_{\Sigma}$. We let
$X = \frac{\partial}{\partial s}\mathcal{F}(s,\cdot)|_{s = 0} \in \Gamma(\Sigma,TM)$ be the variation
field, and we assume for convenience that $X \perp T\Sigma$, so that $X \in \Gamma(\Sigma,N\Sigma)$. This involves no loss of generality,
since tangential components of the variation would not alter the volume in any event. See \cite{gmt22} for a more detailed development of this
argument. We let $\Sigma_s = \mathcal{F}(s,\Sigma)$ be the image of the variation at time $s$.

Before we compute the derivative of $\mathcal{E}$ or $V$, it is necessary to understand how $t$ and $u$ vary with $s$.

\subsection{Distance}\label{smoothdistdersec} 
Let $t_s: M \to [0,\infty)$ be the distance to $\Sigma_s$ with respect to $g$. We will show that, for $p$ near enough (but not on) $\Sigma$, $t_s$ is smooth in $s$ for $s$ very small, and will
also compute the derivative.

Because $M$ is compact, there is a positive lower bound to the injectivity radius, and we assume (by restricting if necessary) that our tubular neighborhood $\mathcal{T}$ is so small that its radius is less than this
lower bound.
\begin{theorem}
	\label{tvarthm}
	Let $||R|| = \sup_{x \in \Sigma}|\riem(x)|_g$ and $||\mathfrak{L}|| = \sup_{x \in \Sigma}|\mathfrak{L}(x)|_g$ be the largest norms (on $\Sigma$) of the ambient curvature tensor and the second fundamental form, respectively.
	Suppose that $p \in \mathcal{T}$ and that $0 < t(p) < \frac{1}{2((n + k)||R|| + ||\mathfrak{L}||)}$. 
	Then there exists $\delta_p \in (0,\delta)$ so that $t_s(p)$ is a smooth function in $s$ on $(0,\delta_p)$.

	Moreover, if $q \in \Sigma$ is the closest point to $p$ and $Y \in N_q\Sigma$ with $|Y|_g = 1$ and $p = \exp_q(t_s(p)Y)$, then
	\begin{equation}\label{tderiveq}
		\left.\frac{\partial t_s(p)}{\partial s}\right|_{s = 0} = -\langle Y,X_q\rangle.
	\end{equation}
\end{theorem}
It is interesting to note that the right-hand side of (\ref{tderiveq}) is independent of $t(p)$, just as in the Euclidean case.

\begin{proof}
	We fix $p$ satisfying the hypothesis, and let $\rho:M \to [0,\infty)$ be the $g$-distance to $p$. 
	Let $q$ be the (unique) closest point on $\Sigma$ to $p$. We claim the Hessian of $\rho$ at $q$ with respect to the induced metric on $\Sigma$ is strictly positive definite.
	To see this, let $Z \in T_q\Sigma$ with $|Z|_g = 1$, and let $\beta:\mathbb{R} \to \Sigma$ be a smooth path with $\beta(0) = q$ and $\beta'(0) = Z$. We wish to show that 
	\begin{equation}
		\label{secderivneg}
		\frac{d^2}{d\tau^2}(\rho(\beta(\tau))) > 0.
	\end{equation}
	Let $\gamma:[0,1] \to M$ be the minimizing geodesic from $p$ to $q$ of speed $\rho(q)$, and let $\Gamma:\mathbb{R} \times [0,1] \to M$ be defined by letting
	$\Gamma(\tau,\cdot)$ be the minimizing geodesic from $p$ to $\beta(\tau)$ with speed $\rho(\beta(\tau))$. Then (\ref{secderivneg}) is equivalent to the positivity of the $\tau$-derivative of the
	length of the curves $\gamma_{\tau}: \lambda \mapsto \Gamma(\tau,\lambda)$. Let $V(\lambda) = \partial_{\tau}\Gamma(\tau,\lambda)|_{\tau = 0}$, which is a Jacobi field on $\gamma$.
	By a standard formula (see e.g. Problem 10-9 of \cite{leeIRM} or \cite{amb61}),
	\begin{equation*}
		\left.\frac{d^2}{d\tau^2}(\rho(\beta(\tau)))\right|_{\tau = 0} = \left.\frac{d^2}{d\tau^2}L(\gamma_{\tau})\right|_{\tau = 0} = \left.\langle D_{\lambda}V,V\rangle_g\right|_{t = 1} +
				\langle\mathfrak{L}(Z,Z),\gamma'(1)\rangle_g.
	\end{equation*}
	Now, the second term is bounded by $||\mathfrak{L}||\rho(q) = ||\mathfrak{L}||t(p)$. We wish to analyze the first. Notice that $V(1) = Z$. Meanwhile, since a Jacobi field is determined by its value at two points,
	$D_\lambda V(1)$ will be the same for any Jacobi field along $\gamma$ that vanishes at $p$ and is $Z$ at $q$.

	Consider now the metric $\tilde{g} = \exp_p^*g$ on $T_pM$. Suppose $v \in T_pM$ is the unique smallest vector such that $\exp_p(v) = q$, and also that $w \in T_v(T_pM) \cong T_pM$ is the unique vector such that
	$(d\exp_p)_v(w) = Z$. Define $\widetilde{\Gamma}:\mathbb{R} \times \mathbb{R} \to T_pM$ by $\widetilde{\Gamma}(\tau,\lambda) = \lambda(v + \tau w)$. 
	Then $\exp_p\circ\Gamma$ is another variation through geodesics about $\gamma$, with the same
	Jacobi field $V$. Since $\exp_p$ is an isometry from a neighborhood of $0 \in T_pM$ to a neighborhood including $q$ in $M$, we may thus compute $\langle D_{\lambda}V,V\rangle$ on $T_pM$ using the radial
	$\tilde{g}$-geodesic variation $\widetilde{\Gamma}$. The Jacobi field may then be written $\widetilde{V}(\lambda) = \lambda w$. Thus,
	$\widetilde{V}'(1) = w + D_{\partial_{\lambda}}^{\tilde{g}}w = w + w^{A}\Gamma_{A\lambda}^B\partial_B$ (where, momentarily, we let $A,B$ denote indices on a set of radial coordinates of $T_pM$).
	Now, by the standard Taylor expansion of normal coordinates at a point, the mean value theorem, and the fact that $|\partial_{\lambda}|_{\tilde{g}} \equiv t(p)$ along $\exp_p^*\gamma$ by Gauss's Lemma,
	we have $|\Gamma_{A\lambda}^B| \leq ||R||t(p)$. Thus,
	\begin{equation*}
		\left.\langle D_{\lambda}V,V\rangle_g\right|_{t = 1} = \langle w, w - D_{\partial_{\lambda}}^{\tilde{g}}w\rangle \geq 1 - (n + k)||R||t(p).
	\end{equation*}
	It now follows from our hypotheses that $\left.\frac{d^2}{d\tau^2}(\rho(\beta(\tau)))\right|_{\tau = 0} > \frac{1}{3}$. We note that the positive definiteness of this Hessian is the only place in the proof where
	the hypothesis on $\rho(p)$ is used.

	For $s \in (-\delta,\delta)$, define $\mathcal{F}_s:\Sigma \to M$ by $\mathcal{F}_s(x) = \mathcal{F}(s,x)$.
	We now choose a trivialization of $T^*\Sigma$ over a neighborhood $U$ of $q$, say $\pi:T^*U \to \mathbb{R}^n$. Define $f:(-\delta,\delta) \times U \to \mathbb{R}^n$ by $f(s,x) = \pi(\mathcal{F}_s^*d\rho)$.
	We claim that $df_x|_{(0,q)}:T_q\Sigma \to T_{f(0,q)}\mathbb{R}^n$ is an isomorphism; indeed, it is easy to verify that this is equivalent to the invertibility of the Hessian of $\rho$, which we have just shown is
	positive definite.

	It now follows from the implicit function theorem that the equation $f(s,\sigma(s)) = 0$ defines a smooth path $\sigma:(-\delta',\delta') \to \Sigma$ satisfying $\sigma(0) = q$, 
	for some $\delta' \in (0,\delta]$. But this equation simply states that at
	$\mathcal{F}(s,\sigma(s))$, $\grad \rho$ will be orthogonal to $T\Sigma_s$. Thus, by Lagrange multipliers, $\mathcal{F}(s,\sigma(s))$ is the closest point on $\Sigma_s$ to $p$. But
	\begin{equation}
		\label{tseq}
		t_s(p) = \rho(\mathcal{F}(s,\sigma(s))).
	\end{equation}
	Thus, $t_s(p)$ is differentiable in $s$ for sufficiently small $s$.

	We may use (\ref{tseq}) to compute the $s$-derivative. Recalling that $\mathcal{F}(0,\cdot) = \id_{\Sigma}$, we have
	\begin{align*}
		\left.\frac{\partial t_s}{\partial s}\right|_{s = 0} &= \left.d\rho(X_q + \sigma'(s))\right|_{s = 0}\\
			&= \langle \grad_g\rho|_q,X_q\rangle + \langle
			\grad_g\rho|_q,\sigma'(0)\rangle.
	\end{align*}
	The second term is zero, by (again) Lagrange multipliers.
	Also, $\grad_g\rho = -Y$, since the former is the direction of
	fastest-increasing growth in $\rho$. Thus, (\ref{tderiveq}) follows.
\end{proof}

\subsection{Singular Yamabe functions}\label{smoothussec} In order to compute variations, we will also require that $u$ itself, the singular Yamabe solution, be differentiable with respect to $s$, and that its derivative be polyhomogeneous.
The proof that this is so is highly involved; fortunately, it is already contained in the wide-ranging paper \cite{ms91}. Because the authors there were concerned with a slightly different question, and set the matter up differently,
it is not immediately transparent that this claim is proved in their paper. We therefore rehearse their setup from our context. The following is simply a recounting of what was done in that paper.

Let $u_s$ be the solution to the singular Yamabe problem for $(M,\Sigma_s)$; thus, $u_s$ is a defining function for $\Sigma_s$, and $g_s^+ = u_s^{-2}g$ has constant scalar curvature
$-(n - k + 2)(n + k - 1)$. For each $s$, we extend the embedding $\mathcal{F}(s,\cdot)$ to a diffeomorphism $\tau_s = \tau(s):M \to M$, depending smoothly on $s$. In particular, we require $\tau_0 = \id$.

Now, let $\widehat{U}_s = u_s^{-(n+ k - 2)/2}$. Thus, $\widehat{U}_s$ satisfies the PDE
\begin{multline}\label{Uhateq}
	\Delta_g\widehat{U}_s - \frac{n + k - 2}{4(n + k - 1)}R_g\widehat{U}_s + \frac{(k - 2)^2 - n^2}{4}\widehat{U}_s^{1 + \frac{4}{n + k - 2}} = 0\\
	\widehat{U}_s \text{ tends to infinity near } \Sigma_s = \tau_s(\Sigma)\\
	\widehat{U}_s^{\frac{4}{n + k - 2}}g \text{ is complete on } M \setminus \Sigma_s.\\
\end{multline}
Like the equation
\begin{equation}
	L[u_s] = 0, \qquad u_s|_{\Sigma_s} = 0,
\end{equation}
to which it is equivalent, (\ref{Uhateq}) has the inconvenient feature that the singularity of its solution moves with $s$. We may transfer the equation to $M \setminus \Sigma$ by setting $U_s = \tau_s^*\widehat{U}_s$.
By the naturality of the conformal Laplacian $\Box_g = \Delta_g - \frac{n + k - 2}{4(n + k - 1)}R_g$, we have
\begin{equation}
	\label{Ueq}
	\Delta_{\tau_s^*g}U_s - \frac{n + k - 2}{4(n + k - 1)}R_{\tau_s^*g}U_s + \frac{(k - 2)^2 - n^2}{4}U_s^{1 + \frac{4}{n + k - 2}} = 0.
\end{equation}
More generally, we might write $U_{\tau}$ for the solution of (\ref{Ueq}) for any diffeomorphism $\tau$. Write $U = U_0(1 + W)$ and $g^+(\tau) = U_0^{4/(n + k - 2)}g$. Then by the conformal change formula
for the conformal Laplacian, we have
\begin{equation}
	\begin{split}
		0 =& \Delta_{g^+(\tau)}(1 + W) - \frac{n + k - 2}{4(n + k - 1)}R_{g^+(\tau)}(1 + W)\\
		&\qquad\qquad+ \frac{(k - 2)^2 - n^2}{4}(1 + W)^{1 + \frac{4}{n + k - 2}} =: H_{\tau}[W].
	\end{split}
\end{equation}
The goal is to use the Implicit Function Theorem on Banach spaces to realize $u_s$ as a smooth path by writing $U_s = U_0(1 + W_s)$ for a smooth path $W_s$; and to this end, we consider solutions $(\tau,W)$ to
$H_{\tau}[W] = 0$ that are near $(\id,0)$.

First, for a given $\tau$, the derivative (in $W$) of $H_{\tau}[W]$ is
\begin{equation*}
	\begin{split}
		E(\tau,W)[\phi] =& \left( \Delta_{g^+(\tau)} - \frac{n + k - 2}{4(n + k - 1)}R_{g^+(\tau)}\right.\\
		&\quad+ \left.\frac{((k - 2)^2 - n^2)(n + k + 2)}{4(n + k - 2)}(1 + W)^{\frac{4}{n + k - 1}} \right)\phi.
	\end{split}
\end{equation*}
Since $R_{g^+(\id)} = R_{g^+} = -(n + k - 1)(n - k + 2)$, we have in particular that
\begin{equation}
	\label{E0}
	E(\id,0)[\phi] = (\Delta_{g^+} + (n - k + 2))\phi.
\end{equation}
Much of section 4 of (\cite{ms91}) is taken with showing that, in contexts including ours here, this map is an isomorphism between appropriate Banach spaces. (It is relevant at this point that we are in
the case $n - k + 2 > 0$.) The authors define Banach spaces of diffeomorphisms and of smooth functions, and use the implicit function theorem (for Banach spaces) applied to the equation
$H_{\tau(s)}[G(\tau(s))] = 0$ to show that $W_s = G(\tau(s))$ depends smoothly on $\tau$, and thus on $s$. This argument is contained in section 2 of that paper, culminating in Theorem 2.21. Note that, although that theorem
is stated and proved for $\spheres^n \subset \spheres^{n + k}$, the arguments leading up to it are explicitly general, and the proof of the corollary itself is also (implicitly) general. 
Also, the kernel $\mathcal{J}(\alpha,\nu)$ mentioned
in their theorem may be taken to be trivial in the negative-curvature case we study here (i.e., because $n - k + 2 > 0$).

Thus, for each $p$, there is some small interval (in $s$) for which $W_s(p)$ is smooth in $s$; and in particular, $\frac{\partial W_s}{\partial s}|_{s = 0}$ exists globally and is in $C^{2,\alpha}(M)$ (as this is the
Banach space in which the derivative takes its value in the argument of \cite{ms91}). It immediately follows that, likewise, $\frac{\partial u_s}{\partial s}|_{s = 0}$ likewise exists globally and is $C^{2,\alpha}$.
We may thus differentiate the equation $L[u_s] = 0$ with respect to $s$, and set $s = 0$, obtaining a linear PDE which this derivative satisfies. But the equation in question is linear of edge type, and so
the usual theory of differential edge operators (\cite{maz91,maz91sy}) may be applied to show that, in fact, $\frac{\partial u_s}{\partial s}|_{s = 0}$ is polyhomogeneous.

%% file: tex/var.tex
\section{Variation Formulae}\label{varsec}
In this section, we prove Theorem \ref{varthm}. The argument draws heavily on that in \cite{g17}.
We let $u_s$ be the solution at time $s$ to the singular Yamabe problem on $(M,\Sigma_s)$, and write $g_s^+ = u_s^{-2}g$. Since $u_0 = u$, of course $g_0^+ = g^+$.
For any function or time-dependent constant $f$, we write $\dot{f} = \partial_sf|_{s = 0}$. We let $X$ and $\mathcal{F}$ be as in section \ref{varsmoothsec}.

\begin{proof}[Proof of Theorem \ref{varthm}]
	Suppose first that $k > 1$. Let $t_s$ be the $g$-distance to $\Sigma_s$. We have
	\begin{equation}
		\label{usexpeq}
		\begin{split}
			u_s =&\ t_s + t_s^2(u_s)_2 + \cdots + t_s^{n + 1}(u_s)_{n + 1} + t_s^{n + 1 + \delta}(u_s)_{n + 1 + \delta}\\
			&\qquad+ t_s^{n + 2}(\log t_s)\mathcal{L}_s + t_s^{n + 2}(u_s)_{n + 2} + o(t_s^{n + 2}).
		\end{split}
	\end{equation}
	Here, as mentioned in section \ref{smoothsec}, $\mathcal{L}_s = 0$ if $n$ is odd and $k > 1$.
	Now, it is clear from (\ref{volexpintroeq}) that
	\begin{equation*}
		\left( \vol_{g_s^+}\left(\left\{ t_s > \varepsilon \right\}\right) \right)^{\sbt} = \dot{c}_0\varepsilon^{-n} + \cdots + \dot{c}_{n - 1}\varepsilon^{-1} + \dot{\mathcal{E}}_{n,k}\log\left( \frac{1}{\varepsilon} \right)
		+ \dot{V}_{n,k} + o(1).
	\end{equation*}
	That is, $\dot{\mathcal{E}}_{n,k}$ and $\dot{V}_{n,k}$ are respectively the logarithmic and constant terms, in $\varepsilon$, in the derivative of the volume expansion.
	Now,
	\begin{equation*}
		\vol_{g_s^+}\left(\left\{ t_s > \varepsilon \right\}\right) = \int_{\left\{ t_s > \varepsilon \right\}} dV_{g_s^+} = \int_{\left\{ t_s > \varepsilon \right\}}u_s^{-n - k}dV_g.
	\end{equation*}
	Also, $u_s = u + s\dot{u} + o(s)$ by the discussion in section \ref{smoothussec}. Define $\omega_s = -\log \frac{u_s}{u}$. It follows that
	\begin{align*}
		\vol_{g_s^+}(\left\{ t_s > \varepsilon \right\}) &= \int_{\left\{ t_s > \varepsilon \right\}}\left( 1 - (n + k)s\frac{\dot{u}}{u} \right)u^{-n - k}dV_{g} + o(s)\\
		&= \int_{ \left\{ t_s > \varepsilon \right\}}(1 + (n + k)s\dot{\omega})dV_{g^+} + o(s).
	\end{align*}
	We may, of course, regard $t_s$ as a function $t_s = t_s(s,t,x)$, where $x \in \widetilde{\Sigma}$.
	It follows from the implicit function theorem that we may define a function $\varepsilon_s:(-\delta,\delta)_s \times \widetilde{\Sigma}\to \mathbb{R}$ so that $t_s(s,\varepsilon_s(s,x),x) = \varepsilon$; and then
	$t_s > \varepsilon$ is equivalent to $t > \varepsilon_s$.
	Thus,
	\begin{equation*}
		\vol_{g_s^+}\left(\left\{ t_s > \varepsilon \right\}\right) = \int_{\left\{ t > \varepsilon_s \right\}}(1 + (n+k)s\dot{\omega})dV_{g^+} + o(s).
	\end{equation*}
	Consequently,
	\begin{equation}\label{doteq}
		\left( \vol_{g_s^+}\left(\left\{ t_s > \varepsilon \right\}\right) \right)^{\sbt} = \int_{\left\{ t > \varepsilon \right\}}(n + k)\dot{\omega}dV_{g^+} - \int_{\widetilde{\Sigma}}(\dot{\varepsilon}_s\vartheta|_{t = \varepsilon})
		dV_{\widetilde{\Sigma}},
	\end{equation}
	where we have used the Leibniz integral rule.

	Observe that $u_s = e^{-\omega_s}u$, so that $g_s^+ = e^{2\omega_s}g^+$. By differentiating the conformal change formula for scalar curvature
	\begin{equation*}
		R_{g_s^+} = e^{-2\omega_s}\left( R_{g^+} - 2(n + k - 1)\Delta_{g^+}\omega_s - (n + k - 1)(n + k - 2)|d\omega_s|_{g_s^+}^2 \right)
	\end{equation*}
	and setting $s = 0$, we have
	\begin{equation*}
		\dot{R}_{g_s^+} = -2(R_{g^+}\dot{\omega} + (n + k - 1)\Delta_{g^+}\omega).
	\end{equation*}
	Now, $R_{g_s^+} \equiv -(n  + k - 1)(n - k + 2)$, so we conclude that
	\begin{equation*}
		\dot{\omega} = \frac{1}{n - k + 2}\Delta_{g^+}\dot{\omega}.
	\end{equation*}
	The first term in (\ref{doteq}) is
	\begin{align*}
		\int_{\left\{ t > \varepsilon \right\}}(n + k)\dot{\omega}dV_{g^+} &= \frac{n + k}{n - k + 2}\int_{ \left\{ t > \varepsilon \right\}}\Delta_{g^+}\dot{\omega}dV_{g^+}\\
		&= \frac{n + k}{n - k + 2}\int_{ \left\{ t = \varepsilon \right\}}-u \partial_t(\dot{\omega})dA_{g^+},
	\end{align*}
	where $dA_{g^+} = u\vartheta|_{t = \varepsilon}dV_{h_0}dV_{\bo}$ and $u\partial_t$ is the inward unit normal to the set $\left\{ t > \varepsilon \right\}$.

	Since $\partial_t\dot{\omega} = u^{-2}(\dot{u}\partial_tu - u\partial_t\dot{u})$, we have
	\begin{equation}
		\label{middleq1}
		\int_{{t > \varepsilon}}(n + k)\dot{\omega}dV_{g^+} = -\frac{n + k}{n - k + 2}\int_{t = \varepsilon}(\dot{u}\partial_tu - u\partial_t\dot{u})\vartheta|_{t = \varepsilon}dV_{h_0}dV_{\bo}.
	\end{equation}
	Now, observe that we may write
	\begin{equation}\label{allexpeq}
		\begin{split}
			\vartheta &= t^{- n - 1}\left( 1 + P - (n + k) t^{n + 1}(\log t)\mathcal{L} - t^{n + 1}((n + k)u_{n + 2} + A) \right) + o(1)\\
			\dot{u} &= \dot{t} + F + (n + 2)t^{n + 1}(\log t)(\dot t \mathcal{L}) + t^{n + 1}(\dot{t}\mathcal{L} + (n + 2)\dot{t}u_{n + 2}) + o(t^{n + 1})\\
		\partial_tu &= 1 + G + (n + 2)t^{n + 1}(\log t)\mathcal{L} + t^{n + 1}(\mathcal{L} + (n + 2)u_{n + 2}) + o(t^{n + 1})\\
		\partial_t\dot{u} &= 2\dot{t}u_2 + H + (n + 1)(n + 2)t^n(\log t)(\dot{t}\mathcal{L}) + t^n((2n + 3)\dot{t}\mathcal{L}\\
		&\qquad\qquad\qquad\qquad\qquad\qquad\qquad+ (n + 1)(n + 2)\dot{t}u_{n + 2}) + o(t^n),
		\end{split}
	\end{equation}
	where $F,G,H,P$ are polynomials in $t$ with coefficients in $C^{\infty}(\widetilde{\Sigma})$ and are $O(t)$, with $F, G$, and $P$ of degree $n$ and $H$  of degree $n - 1$. Moreover,
	since $\dot{t}|_{t = \varepsilon} \in \mathcal{H}_1(\widetilde{\Sigma})$ by (\ref{tderiveq}), $F,G, H$, and $P$ have $t^j$-coefficients in
	$\mathcal{Y}_{j + 1}(\widetilde{\Sigma}),\mathcal{Y}_j(\widetilde{\Sigma}),\mathcal{Y}_{j + 2}(\widetilde{\Sigma})$, and $\mathcal{Y}_j(\widetilde{\Sigma})$, respectively.
	Also, $A \in \mathcal{Y}_{n + 1}(\widetilde{\Sigma})$.

	Suppose $n$ is even and that $n - k + 2 \neq 0$. From (\ref{usexpeq}) and (\ref{allexpeq}), it is straightforward to compute that the $\log\left( \frac{1}{\varepsilon} \right)$-coefficient in (\ref{middleq1}) is
	$\frac{(n + k)(n^2 + 2n + k - 2)}{n - k + 2}\dot{t}\mathcal{L}$, while the same coefficient in $\dot{\varepsilon}_s\vartheta|_{t = \varepsilon}$ is $(n + k)\dot{\varepsilon}_s\mathcal{L}$.

	Note that $\dot{\varepsilon}_s = -\dot{t}|_{t = \varepsilon}$. We may thus conclude from (\ref{doteq}) that
	\begin{equation}
		\dot{\mathcal{E}}_{n,k} = \frac{-(n + k)(n^2 + n + 2k - 4)}{n - k + 2}\int_{\widetilde{\Sigma}}\mathcal{L}\dot{t}|_{t = \varepsilon}dV_{h_0}dV_{\bo}.
	\end{equation}
	But this is equal to
	\begin{equation*}
		\dot{\mathcal{E}}_{n,k} = \frac{-(n + k)(n^2 + n + 2k - 4)}{n - k + 2}\vol_{\bo}(\spheres^{k - 1})\int_{\Sigma}\pi_0(\mathcal{L}\dot{t}|_{t = \varepsilon})dV_{h_0}.
	\end{equation*}
	Since $\mathcal{L} \in \mathcal{H}_1(\widetilde{\Sigma})$ unless $n$ is odd and $k = 1$, the result now follows from Theorem \ref{tvarthm} and Lemma \ref{oneformlem}.

	Now, for odd $n$ and $n - k + 2 > 0$, by parity considerations, $\mathcal{L} \equiv 0$ unless $k = 1$.

	Continuing to take $n$ odd, and assuming $k > 1$ for now, it is straightforward to see that the constant term in (\ref{middleq1}) is
	\begin{equation*}
		\frac{(n + k)(n^2 + 2n + k - 2)}{n - k + 2}\int_{\widetilde{\Sigma}}(u_{n + 2} + \mathcal{A})\dot{t}|_{t = \varepsilon}dV_{h_0}dV_{\bo},
	\end{equation*}
	where $\mathcal{A}$ is some function in $\mathcal{Y}_{n + 1}(\widetilde{\Sigma})$. Since $\dot{t}|_{t = \varepsilon} \in \mathcal{H}_1(\widetilde{\Sigma})$, 
	we may conclude that this is
	\begin{equation*}
		\frac{(n + k)(n^2 + 2n + k - 2)}{n - k + 2}\int_{\widetilde{\Sigma}}u_{n + 2}\dot{t}|_{t = \varepsilon}dV_{h_0}dV_{\bo}.
	\end{equation*}
	Also, the constant term in $\dot{\varepsilon}_s\vartheta|_{t = \varepsilon}$ is $-(n + k)\dot{\varepsilon}_s(u_{n + 2} + \mathcal{B})$ for some $\mathcal{B} \in 
	\mathcal{Y}_{n + 1}(\widetilde{\Sigma})$. Thus, we have
	\begin{align*}
		\dot{V}_{n,k} &= \frac{(n + k)(n^2 + n + 2k - 4)}{n - k + 2}\int_{\widetilde{\Sigma}}u_{n + 2}\dot{t}|_{t = \varepsilon}dV_{h_0}dV_{\bo}\\
		&= \frac{(n + k)(n^2 + n + 2k - 4)}{n - k + 2}\vol_{\bo}(\spheres^{k - 1})\int_{\Sigma}\pi_0(u_{n + 2}\dot{t}|_{t = \varepsilon})dV_{h_0}
	\end{align*}
	Now, $\mathcal{L} \in \mathcal{H}_1(\widetilde{\Sigma})$, and $u_{n + 2}$ may be written $u_{n + 2} = \tilde{u}_{n + 2} + \mathcal{A}$ for $\tilde{u}_{n + 2} \in 
	\mathcal{H}_1(\widetilde{\Sigma})$ and
	some function $\mathcal{A}$ that is even on the $(k - 1)$-sphere. Recalling (\ref{tderiveq}), the result now follows by Lemma \ref{oneformlem}.

	Suppose now that $k = 1$. When $n$ is even, the same proof above goes through and we have the same result.

	When $n$ is odd and $k=1$, we no longer have $\mathcal{L}=0$, so in principle it might show up in the variation formula. 
	It still holds that $\dot{t}|_{t=\epsilon}$ and $\mathcal{L}$ are odd, but $\mathcal{L}$ is now even by Corollary \ref{negcol}, 
	and $u_{n+2}$ has an odd component on the $0$-sphere. From \eqref{usexpeq} and \eqref{allexpeq}, we see that the constant term in 
$(\dot{u}\partial_tu - u\partial_t\dot{u})\vartheta$ is 
\[-\big( 
(2n+1)\dot{t} \mathcal{L} + (n^2+2n-1)\dot{t}u_{n+2}
\big) +\mathcal{O}_1,\]
where $\mathcal{O}_1$ is some odd function on the $0$-sphere. Now, $\mathcal{O}_1$ and $\dot{t}\mathcal{L}$ integrate to zero, and we have the constant term in \eqref{middleq1} as 
\[\int_{\widetilde{\Sigma}}\big( 
(n^2+2n-1)\dot{t}u_{n+2}
\big)|_{t=\epsilon} dV_{h_0}dV_{\mathring{b}}.\]
Also, recalling that $\dot{\varepsilon}_s = -\dot{t}|_{t=\epsilon}$, we can see that the constant term in $\dot{\varepsilon}_s\vartheta|_{t = \varepsilon}$ is $(n + 1)\dot{t}|_{t=\epsilon}(u_{n + 2}) + \mathcal{O}_2$, where $\mathcal{O}_2$ is some odd function on the $0$-sphere. It follows from \eqref{doteq}, and the fact that $\mathcal{O}_2$ integrates to zero, that 
\begin{equation*}
\begin{split}
    \dot{V}_{n,1} &= \int_{\widetilde{\Sigma}}\big( 
(n^2+n-2)\dot{t}u_{n+2}
\big)|_{t=\epsilon} dV_{h_0}dV_{\mathring{b}}\\ &= 2\int_{\Sigma}\big( 
(n^2+n-2)\dot{t}u_{n+2}
\big)|_{t=\epsilon} dV_{h_0}.
\end{split}\end{equation*}
The result follows.
\end{proof}
Suppose $n - k + 2 < 0$ with $n$ even, and we wish to consider the change of $\mathcal{E}$ under variations of $\Sigma$. Since we use only formal solutions, we need not appeal to section \ref{smoothussec} for the smoothness
of $\dot{u}_s$; rather, we merely note that it is defined as a finite expansion in $t$ with coefficients that vary smoothly as $\Sigma$ varies. The above proof then still goes through and we obtain the same variation
formula.

%% file: tex/calcs.tex
\section{Calculations}\label{calcsec}
In this section, we compute the previously defined objects in several settings of interest.

\subsection{Equatorial Spheres}
It is possible to compute the renormalized volume or energy (depending on parity) for the equatorial sphere $\spheres^n$ inside any $\spheres^{n + k}$. We get results for any $n$ and $k < n$, including the critical
case, as the logarithm term vanishes here. Of course, the volumes in the nonnegative-curvature, odd-$n$ case are not unique; see \cite{ms91}. 
As mentioned in \cite{ms91}, it is straightforward via stereographic projection from a point on $\spheres^n$ to see that $\spheres^{n + k} \setminus \spheres^n$ is conformal to $\mathbb{R}^{n + k} \setminus \mathbb{R}^n$
with the flat metric. Introducing cylindrical coordinates around $\mathbb{R}^n$, this metric may be written $g_E = |dx|^2 + d\rho^2 + \rho^2\bo$, where $\bo$ is the round metric on $\spheres^{k - 1}$. This metric is conformally related
to $g_1 = \frac{|dx|^2 + d\rho^2}{\rho^2} + \bo$, which is the product metric on $\mathbb{H}^{n + 1} \times \spheres^{k - 1}$; here $\mathbb{H}^{n + 1}$ is hyperbolic space. 
The product metric is complete of constant scalar curvature, and is thus the (or, for $k + n - 2 \leq 0$, \emph{a}) singular Yamabe metric.
Now, the hyperbolic factor is of course isometric to the ball model of hyperbolic space,
$\mathbb{B}^{n + 1}$ with the metric $\frac{4|dx|^2}{(1-|x|^2)^2}$. We follow \cite{g99} in writing $r = \frac{1-|x|}{1+|x|}$;
we obtain the metric
$g^+ = \frac{1}{r^2}\left( \left(\frac{1-r^2}{2}\right)^2\mathring{h} + dr^2 \right) + \bo$; here $\mathring{h}$ is the round metric on $\spheres^{n}$.

Now, $g^+$ has the volume form
\begin{equation*}
	dV_{g^+} = \sqrt{\left( \frac{1}{r^2}\left( \frac{1 - r^2}{2r} \right)^{2n} \right)}dV_{\bo}dV_{\mathring{h}}dr.
\end{equation*}
Integrating this over the region $\left\{ r > \varepsilon \right\}$, we have
\begin{equation*}
	\vol_{g^+}\left\{ r > \varepsilon \right\} = 2^{-n}\vol_{\bo}(\spheres^{k - 1})\vol_{\mathring{h}}(\spheres^n)\int_{\varepsilon}^1r^{-n-1}(1 - r^2)^ndr.
\end{equation*}
Following Graham's argument in \cite{g99}, this gives us
\begin{equation*}
	\mathcal{E}_{n,k} = (-1)^{\frac{n}{2}}\frac{4\pi^{\frac{n + k}{2}}}{\Gamma\left(\frac{n}{2} + 1\right)\Gamma\left( \frac{k}{2} \right)}
\end{equation*}
for $n$ even, and
\begin{equation*}
	V_{n,k} = (-1)^{\frac{n + 1}{2}}\frac{2\pi^{1 + \frac{n + k}{2}}}{\Gamma\left( \frac{n}{2} + 1\right)\Gamma\left( \frac{k}{2} \right)}
\end{equation*}
for $n$ odd.

In particular, $\mathcal{E}_{2,2} = -4\pi^2$ and $V_{1,2} = -4\pi^2$. The latter is the renormalized volume associated to the equatorial unknot in $\spheres^3$.

In all cases, the quantities are critical: the conformal factor $u$ has no logarithmic term, so for $n$ even, the energy is critical. For odd $n$, the metric $\rho^2g_1$ is a smooth metric conformal to the round metric,
and relative to this metric, $u = \rho$, for which $u_{n + 2} = 0$.

\subsection{The case $n = 2$}

We compute explicitly the energies $\mathcal{E}_{2,k}$ for $k \neq 4$.

\subsubsection{Some metric computations}
We begin with some metric computations that hold in any dimension. We work locally in some Fermi coordinate system on $\Sigma$, and choose a coordinate system $\left\{ \omega^{\mu} \right\}$ as well on some part of the
normal sphere bundle. Recall that we let $c_a = \partial_at = \frac{y^a}{t}$, where the $y^a$ are the normal smooth Fermi coordinates. Recalling that $c_a$ is independent of
$t$ in cylindrical coordinates, we observe that we have
\begin{align*}
	\partial_t &= \sum_{a = 1}^k c_a\partial_a\\
	\partial_{\mu} &= t\sum_{a = 1}^kc_{a,\mu}\partial_a,
\end{align*}
where $c_{a,\mu} = \partial_{\mu}c_a$.

Now, we note that, although the pullback metric $\beta^*g$ is degenerate at $\widetilde{\Sigma}$, the pullback $\beta^*\riem$ of the Riemann curvature tensor extends to a smooth tensor up to $\widetilde{\Sigma}$.
To see this, write the metric tensor $\beta^*g$ in cylindrical coordinates (which are smooth coordinates on $\widetilde{M}$) according to (\ref{gpolar}). Recall that the curvature tensor is given by
\begin{align*}
	R_{IJKL} &= \frac{1}{2}\left( \partial^2_{IK}g_{JL} + \partial_{JL}^2g_{IK}
	- \partial^2_{JK}g_{IL} - \partial^2_{IL}g_{JK}\right)\\
	&\quad + g^{PQ}(\Gamma_{IJP}\Gamma_{KLQ} - \Gamma_{ILP}\Gamma_{KJQ}),
\end{align*}
where $\Gamma_{IJK} = \frac{1}{2}(\partial_Ig_{JK} + \partial_{J}g_{IK}
- \partial_kg_{IJ})$. Now, the only terms that might contribute to singularity at the face $\widetilde{\Sigma}$ are terms involving $g^{\mu\nu}$, i.e., $P = \mu$ and $Q = \nu$; these terms have expressions of the form
$t^{-2}b^{\mu\nu}$, with $b$ smooth. But these terms are always multiplied by a pair of Christoffel symbols of the form $\Gamma_{IJ\mu}$, and from (\ref{gpolar}), such terms are always $O(t)$. Thus, the curvature tensor
is smooth up to $\widetilde{\Sigma}$.

Thus, we may evaluate
\begin{align*}
	R(\partial_{\mu},\partial_t,\partial_t,\partial_{\nu}) &= R_{\mu tt\nu}\\
	&= \sum_{1\leq a,b,c,d\leq k}t^2c_{a,\mu}c_{b,\nu}c_cc_dR_{acdb}.
\end{align*}
Now, $R_{acdb}$ is a smooth function ``downstairs'' on $M$, while the $c$ coefficients are all smooth on $\widetilde{M}$. Thus, $R_{\mu tt\nu}$ is $O(t^2)$, and we may define a smooth
tensor field
\begin{equation}\label{req}
	\mfr_{\mu\nu} := \frac{1}{t^2}R_{\mu tt\nu} = \sum_{1 \leq a,b,c,d \leq k}c_{a,\mu}c_{b,\nu}c_cc_dR_{acdb}
\end{equation}
on $\widetilde{M}$. Next, we note that, away from $\widetilde{\Sigma}$, we have $\nabla_{\partial_t}\partial_{\mu} = \Gamma_{t\mu}^I\partial_I = \frac{1}{t}\partial_{\mu} + O^I(1)\partial_I$, where the notation
$O^I(f)\partial_I$ means the coefficients of the vector field are $O(f)$. Thus, we have
\begin{equation}\label{peq}
	\frac{1}{t}R(\partial_{\mu},\partial_t,\partial_t,\nabla_{t}\partial_{\nu}) = \frac{1}{t^2}R_{\mu tt\nu} + O(t) = \mfr_{\mu\nu} + O(t),
\end{equation}
which similarly extends smoothly to $\widetilde{\Sigma}$.

We write $b = \bo + tb^{(1)} + t^2b^{(2)} + \cdots$. Now, $t^2b_{\mu\nu} = \langle\partial_{\mu},\partial_{\nu}\rangle$. By differentiating twice in $t$, we get
\begin{equation*}
	2b_{\mu\nu} + 4tb_{\mu\nu}' + t^2b_{\mu\nu}'' = \langle D_t^2\partial_{\mu},\partial_{\nu}\rangle + \langle\partial_{\mu},D_t^2\partial_{\nu}\rangle + 2\langle D_t\partial_{\mu},D_t\partial_{\nu}\rangle,
\end{equation*}
where primes denote $t$-derivatives and $D_t$ the connection along the radial geodesic. Another differentiation gives
\begin{equation}\label{thirddiff}
	\begin{split}
		6b_{\mu\nu}' + 6tb_{\mu\nu}'' + t^2b_{\mu\nu}''' &= D_t\left( \langle D_t^2\partial_{\mu},\partial_{\nu}\rangle + \langle\partial_{\mu},D_t^2\partial_{\nu}\rangle \right)\\
		&\quad+ 2\left( \langle D_t^2\partial_{\mu}D_t\partial_{\nu}\rangle + \langle D_t\partial_{\mu},D_t^2\partial_{\nu}\rangle \right).
	\end{split}
\end{equation}
Now, away from $t = 0$ (where the metric is degenerate), $\partial_{\mu}$ is a Jacobi field, and so we have
\begin{align*}
	\langle D_t^2\partial_{\mu},\partial_{\nu}\rangle &= -R(\partial_{\mu},\partial_t,\partial_t,\partial_{\nu})\\
	\langle D_t^2\partial_{\mu},D_t\partial_{\nu}\rangle &= -R(\partial_{\mu},\partial_t,\partial_t,D_t\partial_{\nu}\rangle.
\end{align*}
Plugging these into (\ref{thirddiff}) and using (\ref{req}) and (\ref{peq}), we have
\begin{equation}\label{thirddifftwice}
	6b_{\mu\nu}' + 6tb_{\mu\nu}'' + t^2b_{\mu\nu}''' = -2(t^2\mfr_{\mu\nu})'' - 4t\mfr_{\mu\nu} + O(t) = O(t)
\end{equation}
Differentiating (\ref{thirddifftwice}), we see that
\begin{align*}
	12b_{\mu\nu}'' + 8tb_{\mu\nu}''' &= -2(t^2\mfr_{\mu\nu})'' - 4\mfr_{\mu\nu} + O(t)\\
	&= -8\sum_{1\leq a,b,c,d \leq k}c_{a,\mu}c_{b,\nu}c_cc_dR_{acdb} + O(t).
\end{align*}
It follows that $b_{\mu\nu}^{(1)} \equiv 0$ and
\begin{equation*}
	b_{\mu\nu}^{(2)} = -\frac{1}{3}\sum_{1 \leq a,b,c,d \leq k}c_{a,\mu}c_{b,\nu}c_cc_dR_{acdb}.
\end{equation*}
In particular, if $k = 2$, then $b^{(2)} = -\frac{1}{3}\Sec(\partial_{a^1},\partial_{a^2})$.

We now turn to the determinant. We write
\begin{align*}
	\det h \det \alpha &= (\det h_0 \det \bo)(1 + t\gamma_1 + t^2\gamma_2) + O(t^3),\\
	\intertext{where}
	\gamma_1 &= \tr_{h_0}h^{(1)}\\
	2\gamma_2 &= 2\tr_{h_0}h^{(2)} + (\tr_{h_0}h^{(1)})^2 - \tr_{h_0}(h^{(1)}h_0^{-1}h^{(1)})\\
	&\quad+ 2\tr_{\bo}b^{(2)} - 2\tr_{\bo}(a_0^Th_0^{-1}a_0).
\end{align*}
We denote the vector-valued second fundamental form by $\mfl$: so for $X,Y \in T_p\Sigma$, we have $\mfl(X,Y) = (\nabla^g_XY)^{\perp}$. We let $\mfh$ be the mean curvature vector
$\mfh = h_0^{ij}\mfl^a_{ij}\partial_a$. We will work now exclusively on $\widetilde{\Sigma}$, so to simplify notation and allow the Einstein notation, we introduce
$c^a = g^{ab}c_b = c_a$ and $c^{a}_{\mu} = g^{ab}c_{b,\mu}$. We have, along $\Sigma$,
\begin{equation*}
	\nabla_i\partial_a = -\mfl_{ai}{}^l\partial_l + \Gamma_{ia}^c\partial_c.
\end{equation*}
Using a Jacobi field argument as for $b$ above, we find
\begin{align*}
	h^{(1)}_{ij} &= -2c_a\mfl^a{}_{ij}\\
	h_{ij}^{(2)} &= c^ac^b\left( -R_{iabj} + \mfl_{ail}\mfl_{bj}^l + \Gamma_{ia}^c\Gamma_{jbc}\right)\\
	a_{i\nu}^{(0)} &= c^ac^{b}_{\nu}\Gamma_{iab}.
\end{align*}
Thus,
\begin{align*}
	\tr_{h_0}h^{(1)} &= -2c_a\tr_{h_0}\mfl^a\\
	\tr_{h_0}h^{(2)} &= -c^ac^bh_0^{ij}R_{iabj} + c_{a}c_b\langle \mfl^a,\mfl^b\rangle_{h_0} + c^ac^bh_0^{ij}\Gamma_{ia}^c\Gamma_{jbc},
\end{align*}
so
\begin{equation*}
	\gamma_1 = -2c_a\tr_{h_0}\mfl^a.
\end{equation*}
Now,
\begin{align*}
	c^ac^bR_{ab} &= c^ac^bh_0^{ij}R_{iabj} + \bo^{\mu\nu}c^{a}_{\mu}c^{b}_{\nu}c^cc^dR_{acdb}\\
	\intertext{and}
	\tr_{\bo}b^{(2)} &= -\frac{1}{3}\bo^{\mu\nu}c^{a}_{\mu}c^{b}_{\nu}c^cc^dR_{acdb}.
\end{align*}
So
\begin{equation}\label{firstdiff}
	\begin{split}
	2\gamma_2 - \gamma_1^2 &= -2c^ac^bR_{ab} + \frac{4}{3}\langle dc^a,dc^b\rangle_{\bo}c^cc^dR_{acdb} - 2c_ac_b\langle\mfl^a,\mfl^b\rangle_{h_0}\\
	&\quad+2h_0^{ij}c^ac^b\Gamma_{ia}^c\Gamma_{jbc} - 2h_0^{ij}c^ac^b\langle dc^c,dc^d\rangle_{\bo}\Gamma_{iac}\Gamma_{jbd}.
	\end{split}
\end{equation}
Now, from (\ref{dcaeq}) we conclude that the second line of (\ref{firstdiff}) vanishes
and
\begin{equation*}
	\langle dc^a,dc^b\rangle_{\bo}c^cc^dR_{acdb} = (g^{ab} - c^ac^b)c^cc^dR_{acdb} = g^{ab}c_cc_dR_{acdb},
\end{equation*}
where the last equality follows from the equation $c^ac^bc^cc^dR_{acdb} = R_{tttt} = 0$. Thus,
\begin{equation*}
	2\gamma_2 - \gamma_1^2 = -2c^ac^bR_{ab} + \frac{4}{3}g^{ab}c^cc^dR_{acdb} - 2c_ac_b\langle\mfl^a,\mfl^b\rangle_{h_0}.
\end{equation*}
Now, $c^a \in \mathcal{H}_1(\widetilde{\Sigma})$, since it is the restriction of the harmonic polynomial $y^a$ to the unit sphere. We recall from the proof of Lemma \ref{oneformlem} that
$c^ac^b - \frac{1}{k}g^{ab} \in \mathcal{H}_2(\widetilde{\Sigma})$. This implies that
\begin{equation}\label{gammagammaeq}
	\pi_0(2\gamma_2 - \gamma_1^2) = -\frac{2}{k}g^{ab}R_{ab} + \frac{4}{3k}g^{ad}g^{bc}R_{abcd} - \frac{2}{k}|\mfl|^2.
\end{equation}
We also have
\begin{equation}\label{gamma1eq}
	\pi_0(\gamma_1^2) = \frac{4}{k}|\mfh|^2.
\end{equation}
Also note that
\begin{equation}\label{pi0dgamma}
	\pi_0(|d\gamma_1|_{\bo}^2) = (k - 1)\pi_0(\gamma_1^2).
\end{equation}
This is immediate from the fact that, for $F \in \mathcal{H}_1(\widetilde{\Sigma})$,
\begin{equation*}
	\Delta_{\bo}F^2 = 2F\Delta_{\bo}F + 2|dF|_{\bo}^2 = -2(k - 1)F^2 + 2|dF|_{\bo}^2,
\end{equation*}
and that $\pi_0(\Delta_{\bo}F^2) = 0$.

\subsubsection{The Energies}
For $n=2$, by writing $u = tv_0 + t^2v_1 + t^3v_2$, we have  
\begin{equation}\label{gradnorm}
    |du|^2_g = v_0^2 + 4tv_0v_1 + t^2(4v_1^2+6v_0v_2) + t^2|dv_1|_{\mathring{b}}^2 + O(t^3),
\end{equation}
\begin{equation}\label{arr}
    R_g u^2 = t^2v_0^2R_g|_\Sigma + O(t^3).
\end{equation}

For $m\geq 1$, we have 
\begin{equation}\label{tee}
    \Delta_g t^m =m(m-1)t^{m-2} + m (\dd_t \log \sqrt{\det g})t^{m-1},
\end{equation}
\begin{equation}\label{vee}
    t^2\Delta_g v_{m-1} = \Delta_{\mathring{b}}v_{m-1} + \langle d\log\sqrt{\det h}, dv_{m-1}\rangle_{\mathring{b}}+O(t^2).
\end{equation}

Since 
\[\Delta_g (t^mv_{m-1}) = t^m\Delta_g v_{m-1} + v_{m-1}\Delta_g t^m,\]
by putting \eqref{tee} and \eqref{vee} into 
\begin{equation*}
u\Delta_g u = (tv_0+t^2v_1+t^3v_2)(tv_0\Delta_g + t^2\Delta_g v_1 + t^2v_1\Delta_g + t^3\Delta_g v_2 + t^3v_2\Delta_g),
\end{equation*}
we have 
\begin{align}
u\Delta_g u =& v_0^2 (t\dd_t\log\sqrt{\det g}) + 2tv_0v_1 + tv_0\Delta_{\mathring{b}}v_1 + 3tv_0v_1(t\dd_t\log\sqrt{\det g})\nonumber\\
&\quad+ 6t^2v_0v_2  + t^2 v_0\Delta_{\mathring{b}} v_2 + (4v_0v_2t^2 + 2v_1t^2)(t\dd_t\log\sqrt{\det g})\label{lapp}\\ 
&\quad+ 2t^2v_1^2 + t^2v_1\Delta_{\mathring{b}}v_1 +\langle d\log\sqrt{\det h}, dv_{1}\rangle_{\mathring{b}}+O(t^3).\nonumber
\end{align}
Now, since 
\[\det g = t^{2k-2} \det h \det \alpha =t^{2k-2} (\det h_0 \det \mathring{b})(1+t\gamma_1 +t^2\gamma_2),\]
we have 
\begin{equation}\label{logdet}
	\begin{split}
    		t\dd_t \log \sqrt{\det g} &= \frac{t}{2}\dd_t \log\det g\\
		&= (k-1) + \frac{\gamma_1}{2}t + \left(\gamma_2-\frac{\gamma_1^2}{2}\right)t^2 +O(t^3).
	\end{split}
\end{equation}
Putting \eqref{logdet} into \eqref{lapp} while noting that $\sqrt{\det h} = \sqrt{\det h_0} ( 1+\frac{\gamma_1}{2}t) + O(t^2)$ gives
\begin{align}
    u\Delta_g u =& v_0^2(k-1) + t\left(v_0\Delta_{\mathring{b}} v_1 + v_0v_1 (3k-1) + \frac{v_0^2}{2}\gamma_1\right)\nonumber\\
    &\quad + t^2\left(v_0\Delta_{\mathring{b}}v_2 + v_0v_2(4k+2) + 2kv_1^2 + \frac{3}{2}v_0v_1\gamma_1 + \left(\gamma_2-\frac{\gamma_1^2}{2}\right)\right.\label{truelap}\\
    &\quad+ v_1\Delta_{\mathring{b}}v_1 + \frac{v_0}{2}\langle d\gamma_1,dv_1\rangle_{\mathring{b}}\bigg) + O(t^3)\nonumber
\end{align}
Now upon putting \eqref{gradnorm}, \eqref{arr} and \eqref{truelap} into $L[u]$, the equation $2L[u]=0$ becomes
\begin{equation*}
	\begin{split}
		O(t^3) &= 4 - k - v_0^2(4 - k)\\
		&\quad + t(2v_0\Delta_{\bo}v_1 + 2v_0(k - 5)v_1 + v_0^2\gamma_1)\\
		&\quad + t^2\biggl( 2v_0\Delta_{\bo}v_2 + 2(k-4)v_0v_2 - 8v_1^2 + 3v_0v_1\gamma_1 + 2v_1\Delta_{\bo}v_1\\
		&\quad + \left.\langle dv_1,d(v_0\gamma_1 - (k+2)v_1)\rangle_{\bo} + v_0^2(2\gamma_2-\gamma_1^2) + \frac{v_0^2}{k + 1}R_g\right).
	\end{split}
\end{equation*}
Solving this for $v_0,v_1$, and $v_2$ allows us to compute the $\frac{1}{t}$ coefficient in the volume function $\vartheta = u^{-2-k}\sqrt{\frac{\det h \det \alpha}{\det h_0 \det \bo}}$, namely
\begin{equation}\label{thetaeq}
	\begin{split}
		\vartheta_2 =& \frac{1}{8v_0^{k + 4}}(v_0^2(4\gamma_2 - 2\gamma_1^2) + v_0^2\gamma_1^2 - 4(k+2)v_0(\gamma_1v_1 + 2v_2)\\
		&+ 4(k + 2)(k + 3)v_1^2).
	\end{split}
\end{equation}
We have immediately
\begin{align*}
	v_0 &= 1,\\
	v_1 = \frac{v_0}{8}\gamma_1 &= -\frac{1}{4}c_a\tr_{h_0}\mfl^a.
\end{align*}
Now, setting
\begin{equation*}
	2v_0\Delta_{\bo}v_2 + 2(k - 4)v_0v_2 + \frac{(9 - k)v_0^2}{32}\gamma_1^2 + \frac{(6 - k)v_0^2}{64}|d\gamma_1|_{\bo}^2 + v_0^2(2\gamma_2 - \gamma_1^2) + \frac{v_0^2}{k + 1}R_g
\end{equation*}
to zero, solving it for the zeroth-degree term of $v_2$, and invoking (\ref{pi0dgamma}), we have
\begin{equation*}
	2\pi_0(v_2) = \frac{1}{4 - k}\left( \frac{12 + 5k - k^2}{64}\pi_0(\gamma_1^2) + \pi_0(2\gamma_2 - \gamma_1^2) + \frac{\pi_0(R_g)}{k + 1} \right).
\end{equation*}
Thus,
\begin{align*}
	\vartheta_2 &= \frac{1}{8}\bigg( 4(k+2)(k + 3)\frac{1}{64}\gamma_1^2 - 4(k+2)\frac{1}{8}\gamma_1^2 + \gamma_1^2\\
	&\qquad\qquad+ 2(2\gamma_2 - \gamma_1^2) - 4(k+2)(2v_2) \bigg),
	\intertext{so}
	\pi_0(\vartheta_2) &= \frac{1}{8(k - 4)}\left( \frac{k(10 - k)}{4}\pi_0(\gamma_1^2) + 6k\pi_0(2\gamma_2 - \gamma_1^2)
	+ \frac{4(k + 2)}{(k + 1)}\pi_0(R_g)\right).
\end{align*}
By (\ref{gammagammaeq}) and (\ref{gamma1eq}), this becomes
\begin{equation}\label{pi0t2eq}
	\begin{split}
		\pi_0(\vartheta_2) &= \frac{1}{8(4 - k)}\bigg( (k-10)|\mfh|^2\\
		&\qquad+ 12\left( g^{ab}R_{ab} + |\mfl|^2 - \frac{2}{3}R_{ab}{}^{ba} \right) - \frac{4(k + 2)}{(k + 1)}\pi_0(R_g) \bigg).
	\end{split}
\end{equation}
Now, for $\Sigma^2 \hookrightarrow M^{2 + k}$, we have $|\mfl|^2 = |\mathring{\mfl}|^2 + \frac{1}{2}|\mfh|^2$.
Recall that the Schouten tensor $P$ on $M$ is defined by
\begin{equation*}
	P_{AB} = \frac{1}{k}\left( R_{AB} - \frac{R_g}{2(k + 1)}g_{AB} \right).
\end{equation*}
It follows that
\begin{equation*}
	g^{ab}R_{ab} = g^{AB}R_{AB} - g^{ij}R_{ij} = \frac{kR_g}{k + 1} - kh_0^{ij}P_{ij}.
\end{equation*}
A similar calculation, using also the Gauss-Codazzi equation, shows that
\begin{equation*}
	R_{ab}{}^{ba} = \frac{(k - 1)R_g}{k + 1} - 2kh_0^{jk}P_{jk} + R_{h_0} - \frac{1}{2}|\mfh|^2 + |\mathring{\mfl}|^2.
\end{equation*}
Putting all this together, we conclude that
\begin{equation*}
	\mathcal{E}_{2,k} = \frac{\vol_{\bo}(\spheres^{k - 1})}{8(4 - k)}\int_{\Sigma}\left( k(|\mfh|^2 + 4\tr_{h_0}(P|_{T\Sigma})) + 4|\mathring{\mfl}|^2 - 8R_{h_0} \right)dV_{h_0}.
\end{equation*}
The first term gives a multiple of the Willmore energy, and in particular a multiple of the Graham-Witten anomaly \cite{gw99}. The second term is a pointwise invariant, and of course the third gives
a multiple of the Euler characteristic. One can see that, although this is not defined for $k = 4$, the integrand in that case is in fact the pointwise conformal invariant guaranteed by
Theorem \ref{expandthm}.
For the case $k = 1$, a computation from (\ref{pi0t2eq}) following \cite{g17} gives the simple form
\begin{equation*}
	\mathcal{E}_{2,1} = \frac{1}{2}\int_{\Sigma}(|\mathring{\mfl}|^2 - R_{h_0})dV_{h_0}.
\end{equation*}
This agrees with \cite{g17}, except that our value is twice his, since we also integrate $\vartheta_2$ over the 0-sphere.